\documentclass[11pt, oneside]{amsart}

\usepackage[a4paper, total={6.5in, 8in}]{geometry}

\usepackage{amsmath}
\usepackage[foot]{amsaddr}
\usepackage{hyperref}
\usepackage{mathtools}
\usepackage[makeroom]{cancel}
\usepackage{amssymb}
\usepackage{amsthm}
\usepackage{tensor}
\usepackage{enumerate}
\usepackage{cleveref}
\usepackage{bbm}
\usepackage{float}
\usepackage{varwidth}
\usepackage{pictexwd, dcpic}
\usepackage{xcolor}
\hypersetup{
    colorlinks,
    linkcolor={blue!60!black},
    citecolor={red!60!black},
    urlcolor={blue!80!black}
}
\usepackage{graphicx}
\usepackage{fancyhdr}
\usepackage{mathrsfs}
\usepackage{xfrac}
\usepackage[utf8]{inputenc}
\usepackage{tikz}
\usepackage{scalerel,stackengine}
\usepackage{booktabs}
\usepackage{slashed}
\usepackage[style=alphabetic]{biblatex}
\theoremstyle{plain}
\newtheorem{theorem}{Theorem}
\newtheorem*{theorem*}{Theorem}
\newtheorem{lemma}[theorem]{Lemma}
\newtheorem{proposition}[theorem]{Proposition}
\newtheorem{corollary}[theorem]{Corollary}
\newtheorem*{corollary*}{Corollary}

\theoremstyle{definition}
\newtheorem{definition}[theorem]{Definition}

\theoremstyle{remark}
\newtheorem{remark}[theorem]{Remark}

\numberwithin{theorem}{section}
\numberwithin{equation}{section}

\renewcommand{\d}{\mathrm{d}}
\renewcommand{\leq}{\leqslant}
\renewcommand{\geq}{\geqslant}

\renewcommand{\epsilon}{\varepsilon}

\newcommand{\la}{\lesssim}
\newcommand{\ga}{\gtrsim}

\newcommand{\e}{\mathrm{e}}
\newcommand{\slashgrad}{\slashed{\nabla}}
\newcommand{\vertiii}[1]{{\left\vert\kern-0.25ex\left\vert\kern-0.25ex\left\vert #1 
    \right\vert\kern-0.25ex\right\vert\kern-0.25ex\right\vert}}

\allowdisplaybreaks

\definecolor{aoenglish}{rgb}{0.0, 0.5, 0.0} 

\DeclareFontFamily{U}{MnSymbolC}{}
\DeclareSymbolFont{MnSyC}{U}{MnSymbolC}{m}{n}
\DeclareFontShape{U}{MnSymbolC}{m}{n}{
    <-6>  MnSymbolC5
   <6-7>  MnSymbolC6
   <7-8>  MnSymbolC7
   <8-9>  MnSymbolC8
   <9-10> MnSymbolC9
  <10-12> MnSymbolC10
  <12->   MnSymbolC12}{}
\DeclareMathSymbol{\intprod}{\mathbin}{MnSyC}{'270}

\DeclareFontFamily{U}{BOONDOX-calo}{\skewchar\font=45 }
\DeclareFontShape{U}{BOONDOX-calo}{m}{n}{
  <-> s*[1.05] BOONDOX-r-calo}{}
\DeclareFontShape{U}{BOONDOX-calo}{b}{n}{
  <-> s*[1.05] BOONDOX-b-calo}{}
\DeclareMathAlphabet{\mathcalboondox}{U}{BOONDOX-calo}{m}{n}
\SetMathAlphabet{\mathcalboondox}{bold}{U}{BOONDOX-calo}{b}{n}
\DeclareMathAlphabet{\mathbcalboondox}{U}{BOONDOX-calo}{b}{n}

\DeclareMathOperator{\dvol}{dv}

\title{Wave Map Null Form Estimates via Peter--Weyl Theory}

\author{Grigalius Taujanskas}
\address{Trinity Hall, Trinity Lane, Cambridge CB2 1TJ, UK.}
\email{taujanskas@dpmms.cam.ac.uk}

\begin{document}

\begin{abstract}
	We study spacetime estimates for the wave map null form $Q_0$ on $\mathbb{R} \times \mathbb{S}^3$. By using the Lie group structure of $\mathbb{S}^3$ and Peter--Weyl theory, combined with the time-periodicity of the conformal wave equation on $\mathbb{R} \times \mathbb{S}^3$, we extend the classical ideas of Klainerman and Machedon to estimates on $\mathbb{R} \times \mathbb{S}^3$, allowing for a range of powers of natural (Laplacian and wave) Fourier multiplier operators. A key difference in these curved space estimates as compared to the flat case is a loss of an arbitrarily small amount of differentiability, attributable to a lack of dispersion of linear waves on $\mathbb{R} \times \mathbb{S}^3$. This arises in Fourier space from the product structure of irreducible representations of $\mathrm{SU}(2)$. We further show that our estimates imply weighted estimates for the null form on Minkowski space.
\end{abstract}

\maketitle

\setcounter{tocdepth}{1}
\tableofcontents

\section{Introduction}

Motivated by the study of nonlinear geometric wave equations with null structure on curved spacetimes, in this paper we study spacetime estimates for the null form $Q_0(\phi,\psi) = \nabla_a \phi \nabla^a \psi$ for linear waves $\phi$, $\psi$ on $\mathbb{R} \times \mathbb{S}^3$. On flat spacetimes $\mathbb{R}^{1+n}$, the standard argument via energy estimates and the Sobolev embedding theorem \cite{Soggebook} for the local existence for a system of nonlinear wave equations
\begin{equation} \label{general_scalar_wave_equation} \Box_{\mathbb{R}^{1+n}} \phi^I = F^I(\phi, \partial \phi) \end{equation}
requires the initial data $(\phi^I, \partial_t \phi^I)|_{t=0}$ to be in $H^{s+1}(\mathbb{R}^n) \times H^s(\mathbb{R}^n)$ for $s > n/2$. The restriction $s > n/2$ is imposed by the need to embed $H^s(\mathbb{R}^n)$ in $C^0(\mathbb{R}^n)$, but it is not optimal, at least because Sobolev embeddings necessarily lose a small amount of regularity. For the purpose of illustration, let us for the moment specialise to the case $n=3$. In this case the minimum regularity required by the standard method is therefore $H^{5/2 + \epsilon}(\mathbb{R}^3) \times H^{3/2 + \epsilon}(\mathbb{R}^3)$. In \cite{PonceSideris1993}, Ponce and Sideris showed, using Strichartz estimates \cite{Strichartz1977} obtained by Marshall \cite{Marshall1981} and Pecher \cite{Pecher1985}, that the lower bound for the exponent $s$ may be reduced to $s > s(\ell) = \max\{1, (3\ell-5)/(2\ell-2)\}$, where $\ell$ is the growth exponent of $\partial \phi$ in $F(\phi, \partial \phi)$. Their result shows that as $\ell \to \infty$, the above requirement of $s > 3/2$ is approached, but for nonlinearities quadratic in $\partial \phi$, $\ell = 2$ (which are typical of wave equations arising in mathematical general relativity), the better bound of $s > 1$ holds. The structural properties of $F(\phi,\partial \phi)$ therefore play an important role in the question of regularity of solutions to \eqref{general_scalar_wave_equation}. The work \cite{PonceSideris1993} had in fact been inspired by a landmark result \cite{KlainermanMachedon1993,ErratumKlainermanMachedon93} of Klainerman and Machedon, who had been able to show by proving suitable spacetime estimates that the exponent $s=1$ can be achieved for nonlinearities satisfying the so-called \emph{classical null condition}. In fact, for such null nonlinearities the optimal exponent is predicted by scaling and is given by $s > s_{\text{crit}.} = n/2 -1$, an improvement of another half of a derivative. The case of the critical exponent $s=s_{\text{crit}.}$ is usually much more subtle for which well-posedness may or may not hold.

The classical null condition is the requirement that the nonlinearity $F^I$ be a linear combination of the \emph{null forms}\footnote{As the nonlinearities in the Yang--Mills and Maxwell--Klein--Gordon equations are of the type $Q_{\alpha \beta}$, the null forms \eqref{null_forms_ab} are sometimes referred to as Yang--Mills type \cite{KlainermanSelberg2002}. Similarly, the null form $Q_0$ is common to wave map systems, and is occasionally referred to as the wave map null form. This taxonomy is not strict, however, since $Q_0$ for example appears both in the Yang--Mills equations in Lorenz gauge \cite{SelbergTesfahun2016} and the Einstein--Klein--Gordon equations \cite{LeFlochMa2016}.}
\begin{align}
	\label{null_forms_0}
	& Q_0(\phi, \psi) = \nabla_a \phi \nabla^a \psi, \\
	\label{null_forms_ab}
	& Q_{\alpha \beta}(\phi, \psi) = \nabla_\alpha \phi \nabla_\beta \psi - \nabla_\beta \phi \nabla_\alpha \psi.
\end{align}
Klainerman and Machedon's result \cite{KlainermanMachedon1993} that for such a nonlinearity a local solution exists for initial data merely in  $H^2(\mathbb{R}^3) \times H^1(\mathbb{R}^3)$ is in a certain sense sharp: indeed, Ponce and Sideris show that for $\ell \geq 3$ the $H^{s(\ell)+1}(\mathbb{R}^3) \times H^{s(\ell)}(\mathbb{R}^3)$ norm of the initial data does not in general control the time of existence of the solution, and an example due to Lindblad \cite{Lindblad1993,Lindblad1996} shows that for general nonlinearities (e.g. $F(\phi, \partial \phi) = (\partial_t \phi)^2$ for a scalar equation) a local solution may fail to exist for data in $H^{s+1}(\mathbb{R}^3) \times H^{s}(\mathbb{R}^3)$ when $s \leq 1$. On the other hand, for equations with nonlinearity $Q_0$ (called wave map type equations), Klainerman--Machedon \cite{KlainermanMachedon1995a} and Klainerman--Selberg \cite{KlainermanSelberg1997} (see also Bourgain \cite{Bourgain93a,Bourgain93b}) subsequently improved the small data local existence theorem to the full subcritical range, i.e. $s > 1/2$ in three space dimensions.

The null condition additionally plays a role in the \emph{global} existence of solutions to nonlinear wave equations. For a scalar equation with $F(\phi, \partial \phi) = (\partial_t \phi)^2$, although local solutions exist when $s > 1$, John \cite{John1981} showed that they in fact exhibit finite time blow-ups even when the initial data is arbitrarily small, smooth and compactly supported. By contrast, Nirenberg (\cite{Soggebook}, \S5) observed that when $F(\phi, \partial \phi) = Q_0(\phi,\phi)$, the equation can be transformed into a linear equation via the transformation $\Phi = 1 - \e^{-\phi}$, i.e. $\phi = - \log(1-\Phi)$, so that $\phi$ is global provided $\Phi$ remains small enough. This can be guaranteed by making the initial data sufficiently small. More generally, it is known that for sufficiently small, sufficiently regular initial data solutions to nonlinear wave equations with null nonlinearities are global \cite{Klainerman1980,Klainerman1982,Klainerman1986,Christodoulou1986}. The reason Klainerman's null condition is responsible for such dramatic improvement in the behavior of solutions to nonlinear wave equations is tied to rates of decay of the solution along different directions. For the \emph{linear} wave equation with smooth compactly supported data in three space dimensions, the derivatives tangent to the future lightcone $\{ t = r \}$ decay like $t^{-2}$. The derivative normal to the lightcone, $(\partial_t - \partial_r) \phi$, however, decays only like $t^{-1}$. The null condition is precisely the condition that ensures that the nonlinearity $F^I(\phi, \partial \phi)$ (if it is quadratic in $\partial \phi$) consists of terms which have at most one derivative in the direction normal to the lightcone, and hence better decay. Klainerman and Machedon \cite{KlainermanMachedon1993} observed that in Fourier space this manifests as certain cancellations along the lightcone, a fact that will play an important role in the proofs of our main theorems. The null condition has by now been utilised to great effect to prove, for instance, the global existence of finite energy\footnote{Finite energy data corresponds to the exponent $s=0$. For the Yang--Mills and Maxwell--Klein--Gordon equations, however, the nonlinearities are of the form $F = (-\Delta)^{-1/2}Q$, which allows one to apply the results discussed above, with $s=1$ and $F = Q$.} solutions to the Maxwell--Klein--Gordon and Yang--Mills equations \cite{KlainermanMachedon1994,KlainermanMachedon1995,SelbergTesfahun2013,SelbergTesfahun2016} and the almost optimal well-posedness of the Maxwell--Dirac system \cite{DAnconaFoschiSelberg}, to name a few examples. In the context of global solutions to nonlinear wave equations of the type \eqref{general_scalar_wave_equation}, we mention here also the weak null condition of Lindblad and Rodnianski \cite{LindbladRodnianski2003} and Keir's quite general hierarchical weak null condition \cite{Keir2018,Keir2019} based on the notion of H\"ormander's asymptotic systems \cite{Hormander1987,Hormander1997}.

The specific case of wave maps arises as the simplest nonlinear geometric wave equation and plays various roles in mathematical physics. For example, equations for certain gauges of the Yang--Mills equations are wave maps, as are subsystems of Einstein's equations for spacetimes with two Killing vectors. In various branches of high energy physics wave maps appear under the name of nonlinear $\sigma$ models. On flat backgrounds the mathematical theory of wave maps is very well studied: we have already mentioned the works of Bourgain, Klainerman--Machedon and Klainerman--Selberg on subcritical well-posedness in space dimensions $n \geq 2$ (see also \cite{Zhou97a,Zhou97b}). In space dimension $n=1$ subcritical well-posedness was proved by Keel and Tao \cite{KeelTao98}. Around the turn of the century the critical case was the subject of a number of works by Tao, Klainerman--Rodnianski, Tataru \cite{Tao2000a,Tao2000b,KlainermanRodnianski2001,Tataru2001} and many others, culminating in proofs, for reasonable target manifolds, of global well-posedness for small data in homogeneous Sobolev spaces for $n \geq 2$. D'Ancona and Georgiev \cite{DAnconaGeorgiev2005} showed that in the supercritical regime there is ill-posedness in all dimensions; Tao found that there is also ill-posedness for $n=1$ at the critical exponent \cite{Tao2000c}. 

The curved background case, on the other hand, is much less well understood. Geba \cite{Geba2009} was able to obtain the local well-posedness in just subcritical Sobolev spaces on $(\mathbb{R}^{1+n}, g)$, $3 \leq n \leq 5$, with $g$ a small perturbation of the Minkowski metric, building on the work of Smith and Tataru \cite{SmithTataru2005}. Lawrie \cite{Lawrie2012} used Metcalfe and Tataru's \cite{MetcalfeTataru2012} dispersive Strichartz estimates for variable coefficient wave equations to prove critical well-posedness on a class of restricted perturbations of Minkowski space with space dimension $n \geq 4$. Shatah and Struwe's moving frame approach \cite{ShatahStruwe2002} to prove critical well-posedness on flat backgrounds for $n \geq 4$ has more recently been adapted by Lawrie--Oh--Shahsahani \cite{LawrieOhShahshahani2018} to obtain global well-posedness and scattering for wave maps from $\mathbb{R} \times \mathbb{H}^n$, $n \geq 4$, where $\mathbb{H}^n$ is the $n$-dimensional hyperbolic space, with small data in the critical Sobolev norm. Even more recently, almost critical well-posedness on low regularity curved backgrounds of the form $(\mathbb{R}^{1+2}, g)$ was obtained in the work of Gavrus--Jao--Tataru \cite{GavrusJaoTataru2021}, where the authors use \emph{wave packet}\footnote{These are square-summable pieces of a wave which are localized in both space and frequency on the scale of the uncertainty principle \cite{CordobaFefferman1978} and propagate along null directions.} approximations of solutions to the linear wave equation to obtain the required null form estimates. The wave packet expansions of Gavrus--Jao--Tataru appear to play a role similar to that played by the Peter--Weyl Fourier series in the present work, in which we focus on the crucial estimates on the spatially compact background $\mathbb{R} \times \mathbb{S}^3$.

The key observation which allows one to obtain the improved spacetime estimates for the null nonlinearities on $\mathbb{R}^{1+3}$ is the fact that in Fourier space travelling waves exhibit cancellations which depend on the angle between their spacetime frequencies. The proof then proceeds by employing a positive/negative frequency splitting and applying Plancherel's theorem in spacetime\footnote{A physical space approach has also been introduced by Klainerman, Rodnianski and Tao \cite{KlainermanRodnianskiTao2002}.}. Since standard Fourier techniques rely heavily on the structure of the real numbers, it is difficult to adapt these techniques to curved settings. Using the theory of Fourier integral operators, Sogge \cite{Sogge1993} and Georgiev--Schirmer \cite{GeorgievSchirmer1995} were able to obtain the basic local estimate on spacetimes of the form $\mathbb{R} \times K$, $K$ a smooth compact manifold, by working in a small enough neighbourhood to flatten the metric. However, the extended estimates involving Fourier multiplier operators corresponding to the Laplacian and the wave operator, of the type obtained by Foschi and Klainerman \cite{FoschiKlainerman2000}, do not appear to have been studied even when the background is $\mathbb{R} \times \mathbb{S}^3$.

The main novelty of this paper is to introduce a global geometric approach to study spacetime estimates for the null form \eqref{null_forms_0} on backgrounds of the form $\mathbb{R} \times \mathrm{G}$, where $\mathrm{G}$ is a compact Lie group. We specialise to the particular case of $\mathrm{G} = \mathrm{SU}(2)$ and exploit the Lie group structure of $\mathbb{S}^3 \simeq \mathrm{SU}(2)$, together with the observation that solutions to the free conformal wave equation on $\mathbb{R} \times \mathbb{S}^3$ are periodic in time. The continuous Fourier transform is now replaced by discrete Fourier series, with irreducible representations of $\mathrm{SU}(2)$ taking the place of $\e^{ix\cdot \xi}$. We prove estimates in $L^2$ of spacetime for $(1-\slashed{\Delta})^{-\beta_0/2} (2+\Box)^{\beta_w} Q_0(\phi,\psi)$ for a range of exponents $\beta_0$ and $\beta_w$ akin to the flat space estimates of Foschi \& Klainerman \cite{FoschiKlainerman2000}. A consequence of this method, attributable to the lack of dispersion of linear waves on $\mathbb{R} \times \mathbb{S}^3$, is that there is an arbitrarily small loss of regularity in our estimates, which we can trace down to the fact that irreducible representations of $\mathrm{SU}(2)$ are not one-dimensional, i.e. $\mathrm{SU}(2)$ is non-abelian. Our estimates are valid for large times on $\mathbb{R} \times \mathbb{S}^3$, and in particular imply global weighted estimates on $\mathbb{R}^{1+3}$ with respect to a hyperboloidal foliation.

Estimates for $Q_{\alpha \beta}$, like  the ones presented in this paper for $Q_0$, appear in principle obtainable by similar methods. Here, however, we rely on the fact that the null form $Q_0$ is a scalar quantity, and moreover that $Q_0(\phi, \psi)$ is essentially $\Box(\phi \psi)$. This structure is absent for $Q_{\alpha \beta}$. Moreover, a natural geometric interpretation of $Q_{0i}(\phi, \psi)$ and $Q_{ij}(\phi ,\psi)$ is as the 1-form $\partial_t \phi \, \d \psi - \partial_t \psi \, \d \phi$ and 2-form $\d \phi \wedge \d \psi$ respectively, and these are more difficult to define the Peter--Weyl Fourier transform of. One may of course contract these with a basis of left-invariant vector fields on $\mathrm{SU}(2)$ to produce scalar quantities, but these are non-constant, which introduces other complications.

\section{Structure of Paper and Notation}

\subsection{Structure}

The structure of the paper is as follows. In \Cref{sec:main_results} we state our main results. In \Cref{sec:classical_estimate} we then briefly recap the classical argument of \cite{KlainermanMachedon1993} on $\mathbb{R}^{1+3}$. In \Cref{sec:setup_on_cylinder} we set up the necessary theory and in \Cref{sec:basic_estimate_on_cylinder} prove the basic estimate on $\mathbb{R} \times \mathbb{S}^3$. We then generalize the basic estimate to include Fourier multipliers in \Cref{sec:estimates_with_multipliers}. In \Cref{sec:proof_of_first_corollary} we prove a forced version of the basic estimate, and in \Cref{sec:conformal_estimate} we derive a weighted estimate on Minkowski space. Finally we include a brief summary of the basic facts from Peter--Weyl theory in \Cref{sec:harmonicanalygroups}.

\subsection{Notation}

Our spacetime signature is $(+,-,-,-)$. We denote by $\Box$ the wave operator on $\mathbb{R} \times \mathbb{S}^3$ with the metric
\[ g_{ab}^{\mathbb{R} \times \mathbb{S}^3} = \d t^2 - g_{\mathbb{S}^3}, \]
where $g_{\mathbb{S}^3}$ is the standard round metric on $\mathbb{S}^3$. We will denote by $\dvol_{\mathbb{S}^3}$ the corresponding volume form on $\mathbb{S}^3$, and by $\d \mu = (2\pi^2)^{-1} \dvol_{\mathbb{S}^3}$ the normalised Haar measure on $\mathrm{SU}(2)$. To avoid ambiguity when referring to different spacetimes, we will denote by $\Box_{\mathbb{R}^{1+3}}$ the wave operator on $\mathbb{R}^{1+3}$ with the standard Minkowski metric $\eta_{ab}$. Similarly, we will sometimes write $\Box_{\mathbb{R} \times \mathbb{S}^3}$ in place of $\Box$. We denote by $\slashed{\Delta}$ the Laplacian on $\mathbb{S}^3$ and by $\slashed{\nabla}$ the Levi-Civita connection on $\mathbb{S}^3$. On scalars, $\Box$ will therefore be equal to $\partial_t^2 - \slashed{\Delta}$. To differentiate between objects on $\mathbb{R}^{1+3}$ and $\mathbb{R} \times \mathbb{S}^3$, we will denote the Minkowskian objects (i.e. on $\mathbb{R}^{1+3}$) with a tilde, e.g. $\tilde{\phi}$, and objects on $\mathbb{R} \times \mathbb{S}^3$ plainly, e.g. $\phi$. Hence, for example, $\tilde{t}$ and $\tilde{r}$ will denote the standard time and radial coordinates on Minkowski space, and $\tilde{\nabla}$ the Levi-Civita connection on $\mathbb{R}^{1+3}$. We will denote in bold, $\tilde{\boldsymbol{\nabla}}$, the spatial derivatives on Minkowski space, and we will use the notation $\langle x \rangle = (1+|x|^2)^{1/2}$. On both Minkowski space and the cylinder we will work with Fourier transforms. To this end we will denote with a hat, i.e. $\hat{\phi}$ or $\hat{\tilde{\phi}}$, the Fourier transforms in the \emph{space} variables, and with $\mathcal{F}$, e.g. $\mathcal{F}(\phi)$ or $\mathcal{F}(\tilde{\phi})$, the full spacetime Fourier transforms. To avoid ambiguity we will frequently explicitly include the frequency variables, as in $\mathcal{F}_{t\to \tau, x\to \xi}(\tilde{\phi})$. When performing estimates, by $f \la g$ we mean that $f \leq C g$ for some positive constant $C$, which may change from line to line, and may depend on quantities considered fixed, such as the geometry of the background or the Sobolev exponents involved. By $ f \simeq g$ we mean that there exists a constant $C$ such that $C^{-1} g \leq f \leq C g$. When working with Fourier series on $\mathbb{R} \times \mathbb{S}^3$, we assume a priori that the functions involved are smooth so that the series involved converge uniformly. The full function spaces are then recovered at the end by density. For a complex matrix $M$ we denote by $\vertiii{M}^2 = \operatorname{Tr}(M M^\dagger) = \sum_{i,j} |M_{ij}|^2$ its Frobenius norm. 

\subsection{Fractional weighted Sobolev spaces $W^s_\delta$} Via a conformal compactification of $\mathbb{R}^3$, the standard (fractional) Sobolev spaces $H^s$ on $\mathbb{S}^3$ give rise to weighted Sobolev spaces, which we call $W^s_\delta$, on $\mathbb{R}^3$ as follows. Compactify $\mathbb{R}^3$ by 
\[ g_{\mathbb{S}^3} = \omega^2 g_{\mathbb{R}^3}, \]
where $\omega = 2 \langle \tilde{r} \rangle^{-2}$, with $\tilde{r}$ the standard radial coordinate on $\mathbb{R}^3$ and $\langle \tilde{r} \rangle = (1+\tilde{r}^2)^{1/2}$. In terms of angles $(z,\theta,\varphi)$ on $\mathbb{S}^3$, we have $z=2\arctan \tilde{r}$ and $g_{\mathbb{S}^3} = \d z^2 + \sin^2 z (\d \theta^2 + \sin^2 \theta \d \varphi^2)$, with $\omega = 2 \cos^2(z/2)$. The volume forms on $\mathbb{S}^3$ and $\mathbb{R}^3$ are then related by $\dvol_{\mathbb{S}^3} = \omega^3 \dvol_{\mathbb{R}^3} = 8\langle \tilde{r} \rangle^{-6} \dvol_{\mathbb{R}^3}$. We define the Sobolev spaces $H^s(\mathbb{S}^3)$, $s \in \mathbb{R}$, as usual, using the norm $\| f \|_{H^s} = \| (1-\slashed{\Delta})^{s/2} f \|_{L^2}$. For $\tilde{f}: \mathbb{R}^3 \to \mathbb{R}$ we then say 
\begin{equation}
	\label{weighted_spaces_definition}
	 \tilde{f} \in W^s_{\delta}(\mathbb{R}^3) \iff \| \tilde{f} \|_{W^s_\delta(\mathbb{R}^3)} \overset{\text{def}}{=} \| \omega^{-(\delta +3 - s)/2} \tilde{f} \|_{H^s(\mathbb{S}^3)} < \infty ,
\end{equation}
the power of $\omega$ being chosen in such a way that for integer $s$ this definition is compatible with the usual weighted spaces $H^s_\delta(\mathbb{R}^3)$ (cf. \cite{ChoquetBruhatChristodoulou1981,Christodoulou1986,GeorgievSchirmer1995}). Indeed, it is straightforward to check that when $s \in \mathbb{N}_0$,
\[ \|\tilde{f} \|^2_{W^s_\delta(\mathbb{R}^3)} \la \| \tilde{f} \|^2_{H^s_\delta(\mathbb{R}^3)} \overset{\text{def}}{=} \sum_{k=0}^s \int_{\mathbb{R}^3} \langle \tilde{r} \rangle^{2\delta + 2k} |\tilde{\boldsymbol{\nabla}}^k \tilde{f}|^2 \dvol_{\mathbb{R}^3}. \]

\section{Main Results} \label{sec:main_results}

We study solutions $\phi$, $\psi$ of the wave equations
\begin{equation} \label{wave_equations} \hspace{-20pt} \Box_{\mathbb{R} \times \mathbb{S}^3} \phi + \phi = F, \hspace{66pt} \Box_{\mathbb{R} \times \mathbb{S}^3} \psi + \psi = G 	
\end{equation}
with initial data
\begin{equation} \label{initial_data} (\phi, \partial_t \phi)|_{t=0} = (f_0, f_1), \qquad \qquad (\psi, \partial_t \psi)|_{t=0} = (g_0, g_1).
\end{equation}
Recall that the standard energy inequality for linear wave equations \eqref{wave_equations} states (see e.g. \cite{HormanderVol3}, Lemma 23.2.1) that $\forall  \sigma \in \mathbb{R}$
\begin{align}
	\begin{split}
	\label{energy_inequality}
	\| \phi (t,\cdot)\|_{H^{\sigma+1}(\mathbb{S}^3)} &+ \| \partial_t \phi(t,\cdot) \|_{H^\sigma(\mathbb{S}^3)} \la \| f_0 \|_{H^{\sigma+1}(\mathbb{S}^3)} + \| f_1 \|_{H^\sigma(\mathbb{S}^3)} + \int_0^t \| F(s,\cdot) \|_{H^\sigma(\mathbb{S}^3)} \, \d s.
	\end{split}
\end{align}
We will consider the null form 
\[ Q_0(\phi,\psi) \overset{\text{def}}{=} \nabla_a \phi \nabla^a \psi = \partial_t \phi \partial_t \psi - \slashgrad \phi \cdot \slashgrad \psi, \]
but before stating our results we make a few remarks. 

It is natural to consider the operator $\Box_{\mathbb{R} \times \mathbb{S}^3} + 1$, instead of $\Box_{\mathbb{R} \times \mathbb{S}^3}$, for at least two reasons. First, $(\Box_{\mathbb{R} \times \mathbb{S}^3} + 1 )\phi = F$ is the conformal analogue of the equation $\Box_{\mathbb{R}^{1+3}} \tilde{\phi} = \tilde{F}$, where $\tilde{\phi}$ and $\phi$ are related by the rescaling $\tilde{\phi} = \Omega^{-1} \phi$, the Minkowski metric $\eta_{ab}$ is related to the metric $g^{\mathbb{R} \times \mathbb{S}^3}_{ab}$ on the cylinder $\mathbb{R} \times \mathbb{S}^3$ by the conformal transformation $g^{\mathbb{R} \times \mathbb{S}^3}_{ab} = \Omega^2 \eta_{ab}$, and $\tilde{F} = \Omega^3 F$, where $\Omega$ is the conformal factor given in \eqref{conformal_factor}. Second, as we will see, solutions to the free equation $(\Box_{\mathbb{R} \times \mathbb{S}^3} + 1 )\phi = 0$ are periodic in time, which makes the time Fourier variable discrete. This will be important for our estimates since the space Fourier variables on $\mathbb{S}^3$ will also be discrete.

Moreover, it is classical (see the proof of \Cref{cor:inhomogeneous_basic_estimate}) that estimates for the inhomogeneous equations \eqref{wave_equations}--\eqref{initial_data} may be obtained from estimates for \eqref{free_equations_with_time_data_2} below, by applying Duhamel's principle. Therefore consider two solutions $\phi$, $\psi$ to the free equations
\begin{align}
\label{free_equations_with_time_data_2}
\begin{split}
	& \Box_{\mathbb{R} \times \mathbb{S}^3} \phi + \phi = 0, \hspace{61pt} \Box_{\mathbb{R} \times \mathbb{S}^3} \psi + \psi = 0, \qquad \text{with} \\
	& (\phi, \partial_t \phi)|_{t=0} = (0, f), \qquad \qquad (\psi, \partial_t \psi)_{t=0} = (0, g).
\end{split}
\end{align} We prove the following main theorems:

\begin{theorem}[Basic Estimate] \label{thm:basic_estimate} For any $\epsilon > 0$ there exists a universal constant $C = C(\epsilon) > 0$ such that solutions $\phi$, $\psi$ to \eqref{free_equations_with_time_data_2} satisfy the estimate
\[ \| Q_0(\phi, \psi) \|^2_{L^2([-\pi,\pi] \times \mathbb{S}^3)} \leq C \| f \|^2_{H^{1}(\mathbb{S}^3)} \| g \|^2_{H^\epsilon(\mathbb{S}^3)}. \]
Here the time interval $[-\pi, \pi]$ may also be replaced with any other interval, e.g. $[0,T]$ for any $T>0$, when the constant $C = C(\epsilon, T)$ then depends only on $\epsilon$ and $T$.
\end{theorem}

\begin{remark}
	The arbitrarily small loss of differentiability $\epsilon$ on the right-hand side of the above estimate arises, as we will see, from the product structure of Fourier modes on $\mathrm{SU}(2)  = \mathbb{S}^3$. The unitary dual of $\mathrm{SU}(2)$ is not an algebra (in contrast to the abelian case of e.g. $\mathbb{R}^3$), which results in products of irreducible representations of $\mathrm{SU}(2)$ (we will use Wigner's $\mathrm{D}$-matrices, which are closely related to the usual spin-weighted spherical harmonics ${}_s Y_{lm}$) having non-trivial Clebsch--Gordan expansions. It is these expansions which in the estimates interact with H\"older's inequality to produce the above loss. 
\end{remark}

\begin{remark}
	An immediate corollary of \Cref{thm:basic_estimate} for solutions of \eqref{wave_equations}--\eqref{initial_data} with $F = 0 = G$ is the estimate
	\begin{equation}
		\label{basic_estimate_2}
		\| Q_0(\phi, \psi) \|^2_{L^2([-\pi,\pi] \times \mathbb{S}^3)} \la \left( \| f_0 \|^2_{H^2(\mathbb{S}^3)} + \| f_1 \|^2_{H^1(\mathbb{S}^3)} \right) \left( \| g_0 \|^2_{H^{1+\epsilon}(\mathbb{S}^3)} + \|g_1 \|^2_{H^{\epsilon}(\mathbb{S}^3)} \right).
	\end{equation}
\end{remark}

\begin{theorem}[Estimates with Multipliers]
	\label{thm:estimates_with_multipliers}
	For solutions $\phi$, $\psi$ of \eqref{wave_equations}--\eqref{initial_data} there exists a universal constant $C = C(\alpha_1, \alpha_2, \beta_0, \beta_w) > 0$ such that the estimate
	\[ \| J^{-\beta_0} W^{\beta_w} Q_0(\phi,\psi) \|_{L^2([-\pi,\pi] \times \mathbb{S}^3)} \leq C \| f \|_{H^{\alpha_1}(\mathbb{S}^3)} \| g \|_{H^{\alpha_2}(\mathbb{S}^3)}, \]
	where $J = (1 - \slashed{\Delta})^{1/2}$ and $W = (2+ \Box)$, holds provided the following bounds are satisfied:
	\begin{align}
		\label{alpha_1_alpha_2_condition_1} \alpha_1 + \alpha_2 + \beta_0 &> 1 + 2\beta_w, \\
		\label{alpha_1_alpha_2_condition_2} \alpha_1 + \alpha_2 &\geq  1 + 2\beta_w, \\
		\label{alpha_1_alpha_2_condition_4} \alpha_1 + \beta_0 & \geq \beta_w, \\
		\label{alpha_1_alpha_2_condition_5} \alpha_2 + \beta_0 & \geq \beta_w,
	\end{align}
	and additionally
	\begin{equation}
		\beta_w \geq - 1
	\end{equation}
	and
	\begin{equation}
		\label{alpha_1_alpha_2_condition_3} -\frac{1}{2} \leq \beta_0 \leq 2\beta_w + \frac{3}{2}.
	\end{equation}
\end{theorem}

\begin{remark} 
	We remark that \Cref{thm:estimates_with_multipliers} implies \Cref{thm:basic_estimate} if we set $\beta_w = 0$, $\beta_0 = 0$, $\alpha_1 = 1$, and $\alpha_2 = \epsilon$.
\end{remark}

As a proof of concept, we also show that our estimates imply the following corollaries; of course these are not the most general estimates implied by \Cref{thm:estimates_with_multipliers} (in particular we have set $\beta_0 = \beta_w = 0$ in \Cref{cor:weighted_estimates}, and $\beta_w = 0$, $\beta_0 = -\epsilon$ in \Cref{cor:inhomogeneous_basic_estimate}).

\begin{corollary}[Forced Basic Estimate] \label{cor:inhomogeneous_basic_estimate}
For any given $0 < \epsilon \leq \frac{1}{2}$ there exists a universal constant $C(\epsilon)>0$ such that solutions $\phi$, $\psi$ of \eqref{wave_equations}--\eqref{initial_data} satisfy the estimate
	\begin{align*} \| Q_0(\phi, \psi) \|_{L^2([-\pi,\pi]; H^{1+\epsilon}(\mathbb{S}^3))} & \leq C(\epsilon) \left( \| f_0 \|_{H^{2+\epsilon}(\mathbb{S}^3)} + \| f_1 \|_{H^{1+\epsilon}(\mathbb{S}^3)} + \int_{-\pi}^{\pi} \| F(t,\cdot) \|_{H^{1+\epsilon}(\mathbb{S}^3)} \, \d t \right) \\
	& \hspace{25pt} \times \left( \| g_0 \|_{H^{2+\epsilon}(\mathbb{S}^3)} + \| g_1 \|_{H^{1+\epsilon}(\mathbb{S}^3)} + \int_{-\pi}^{\pi} \| G(t,\cdot) \|_{H^{1+\epsilon}(\mathbb{S}^3)} \, \d t \right).
	\end{align*}
As before, the time interval $[-\pi, \pi]$ may also be replaced with any other interval, e.g. $[0,T]$ for any $T>0$, when the constant $C(\epsilon, T)$ then depends only on $\epsilon$ and $T$.
\end{corollary}

\begin{remark}
	In fact, we can prove slightly more, and also control the time derivative of $Q_0$ in $L^2([-\pi,\pi];H^\epsilon(\mathbb{S}^3))$ at the expense of replacing the norms on $F$ and $G$ on the right-hand side with $L^2([-\pi,\pi];H^{1+\epsilon}(\mathbb{S}^3))$ instead of $L^1([-\pi,\pi];H^{1+\epsilon}(\mathbb{S}^3))$. See \cref{main_estimate_with_time_derivative}.
\end{remark}

\begin{corollary}[Weighted Estimate on Minkowski Space] \label{cor:weighted_estimates}
	Consider the linear waves $\Box_{\mathbb{R}^{1+3}} \tilde{\phi} = 0 = \Box_{\mathbb{R}^{1+3}} \tilde{\psi}$ with initial data $(\tilde{\phi}, \partial_{\tilde{t}} \tilde{\phi})|_{\tilde{t}=0} = (\tilde{f}_0, \tilde{f}_1)$ and $(\tilde{\psi}, \partial_{\tilde{t}} \tilde{\psi})|_{\tilde{t}=0} = (\tilde{g}_0, \tilde{g}_1)$, where $\tilde{t}$ is the standard time coordinate on $\mathbb{R}^{1+3}$. Then for any sufficiently small $\epsilon > 0$ the null form 
	\[ \widetilde{Q}_0(\tilde{\phi}, \tilde{\psi}) = \tilde{\nabla}_a \tilde{\phi} \tilde{\nabla}^a \tilde{\psi} = \partial_{\tilde{t}} \tilde{\phi} \partial_{\tilde{t}} \tilde{\psi} - \tilde{\boldsymbol{\nabla}} \tilde{\phi} \cdot \tilde{\boldsymbol{\nabla}} \tilde{\psi} \]
	satisfies the weighted estimate
	\[ \left\| \langle \tilde{t} - \tilde{r} \rangle \langle \tilde{t} + \tilde{r} \rangle \widetilde{Q}_0(\tilde{\phi},\tilde{\psi}) \right\|^2_{L^2(\mathbb{R}^4)} \la \left( \| \tilde{f}_0 \|^2_{W^2_1(\mathbb{R}^3)} + \| \tilde{f}_1 \|^2_{W^1_2(\mathbb{R}^3)} \right) \left( \| \tilde{g}_0 \|^2_{W^{1+\epsilon}_{\epsilon}(\mathbb{R}^3)} + \| \tilde{g}_1 \|^2_{W^{\epsilon}_{1+\epsilon}(\mathbb{R}^3)} \right). \] 
\end{corollary}

\begin{remark}
	The estimate of \Cref{cor:weighted_estimates} is a consequence of the basic estimate of \Cref{thm:basic_estimate}. \Cref{thm:estimates_with_multipliers} will imply similar, more general estimates on $\mathbb{R}^{1+3}$ with respect to a hyperboloidal foliation.
\end{remark}

\section{Classical Estimate on $\mathbb{R}^{1+3}$} \label{sec:classical_estimate}

In this section we briefly review the main elements of the proof in $\mathbb{R}^{1+3}$, as this will help to motivate the strategy on $\mathbb{R} \times \mathbb{S}^3$. For solutions $\tilde{\phi}$, $\tilde{\psi}$ of
\begin{equation} \label{wave_equation_flat} \Box_{\mathbb{R}^{1+3}} \tilde{\phi} = 0 = \Box_{\mathbb{R}^{1+3}} \tilde{\psi} \end{equation}
on $\mathbb{R}^{1+3}$ with data $(\tilde{\phi}, \partial_{\tilde{t}} \tilde{\phi})|_{\tilde{t}=0} = (0, \tilde{f})$, $(\tilde{\psi}, \partial_{\tilde{t}} \tilde{\psi})|_{\tilde{t}=0} = (0, \tilde{g})$, the key result obtained in \cite{KlainermanMachedon1993} is the estimate
\begin{equation}
\label{key_continuous_estimate}
\| \widetilde{Q}(\tilde{\phi}, \tilde{\psi}) \|_{L^2(\mathbb{R}^4)} \leq C \| \tilde{f} \|_{L^2(\mathbb{R}^3)} \| \boldsymbol{\nabla} \tilde{g} \|_{L^2(\mathbb{R}^3)},
\end{equation}
where $\widetilde{Q}$ is any of the classical null forms \eqref{null_forms_0}, \eqref{null_forms_ab} on $\mathbb{R}^{1+3}$. The method to obtain \eqref{key_continuous_estimate} may be broken down into the following steps. For the purpose of exposition we recap the case of $\widetilde{Q}_0(\tilde{\phi}, \tilde{\phi})$. We start with data $\tilde{f} \in \mathcal{S}(\mathbb{R}^3)$ such that its Fourier transform is smooth and compactly supported, and vanishes in a neighbourhood of the origin; these assumptions may be removed at the end by a density argument.

\subsection{Positive/negative frequency splitting} The Fourier transform in space of \eqref{wave_equation_flat} gives the solution $\hat{\tilde{\phi}}(\tilde{t}, \xi) = (\sin (|\xi|\tilde{t})/ |\xi| ) \hat{\tilde{f}}(\xi)$. One considers the positive and negative frequency parts of this solution,
	\[ \hat{\tilde{\phi}}^\pm(\tilde{t}, \xi) = \frac{e^{\pm i |\xi| \tilde{t}}}{|\xi|} \hat{\tilde{f}}(\xi), \]
	or
	\[ \tilde{\phi}^\pm(\tilde{t},\tilde{x}) = \frac{1}{(2 \pi)^3} \int_{\mathbb{R}^3} \frac{e^{\pm i |\xi| \tilde{t}}}{|\xi|} \hat{\tilde{f}}(\xi) e^{i \tilde{x} \cdot \xi} \, \d \xi \]
	in physical space, with the original solution given by $\tilde{\phi} = \frac{1}{2i} (\tilde{\phi}^+ - \tilde{\phi}^-)$. By the bilinearity of $\widetilde{Q}_0$, it is then enough to prove the estimate \eqref{key_continuous_estimate} for every pair $\widetilde{Q}_0^{\pm, \pm} \overset{\text{def}}{=} \widetilde{Q}_0(\tilde{\phi}^\pm, \tilde{\phi}^\pm)$. For brevity here we consider only $\widetilde{Q}_0^{+,+}$.
	
\subsection{Spacetime Fourier transform of $\tilde{\phi}^\pm$} Taking now the spacetime Fourier transform of, say, $\tilde{\phi}^+$ gives
	\[ \mathcal{F}_{\tilde{t} \to \tau, \, \tilde{x} \to \xi} (\tilde{\phi}^+ ) = \mathcal{F}_{\tilde{t} \to \tau} \left( \frac{e^{ i |\xi| \tilde{t}}}{|\xi|} \hat{\tilde{f}}(\xi) \right) = 2 \pi \frac{\hat{\tilde{f}}(\xi)}{|\xi|} \delta(\tau - |\xi|). \]
	Using the homogeneity of $\delta$, this can be rewritten as
	\begin{equation} \label{FT_of_initial_data} 2 \pi \frac{\hat{\tilde{f}}(\xi)}{|\xi|} \delta(\tau - |\xi|) = 4 \pi \hat{\tilde{f}}(\xi) \chi_{[0, \infty)}(\tau) \delta(\tau^2 - |\xi|^2). \end{equation}
	
\subsection{Convolution in spacetime} Using $\Box_{\mathbb{R}^{1+3}} \tilde{\phi} = 0$, one can rewrite the null form as $\widetilde{Q}_0^{+,+} = \nabla_a \tilde{\phi}^+ \nabla^a \tilde{\phi}^+ = \frac{1}{2} \Box_{\mathbb{R}^{1+3}} (\tilde{\phi}^+)^2$. The spacetime Fourier transform of $\widetilde{Q}_0^{+,+}$ is then
	\begin{equation} \label{inverse_convolution_flat_space} \mathcal{F}_{\tilde{t} \to \tau, \, \tilde{x} \to \xi}(\widetilde{Q}_0^{+,+}) = \frac{1}{2}(\tau^2 - |\xi|^2) \mathcal{F}_{\tilde{t} \to \tau, \, \tilde{x} \to \xi} ((\tilde{\phi}^+)^2)  = \frac{1}{2} (\tau^2  - |\xi|^2 ) \mathcal{F}_{\tilde{t} \to \tau, \, \tilde{x} \to \xi}(\tilde{\phi}^+) * \mathcal{F}_{\tilde{t} \to \tau, \, \tilde{x} \to \xi}(\tilde{\phi}^+). \end{equation} 
	To compute this convolution, one has the following lemma (the proof may be found in \cite{KlainermanMachedon1993}, p. 1231).
	\begin{lemma} \label{lemma:spacetime_convolution_lemma} The spacetime Fourier transform of $(\tilde{\phi}^+)^2$ is given by
	\[ \mathcal{F}_{\tilde{t} \to \tau, \, \tilde{x} \to \xi} ( (\tilde{\phi}^+)^2) = \frac{2 \pi^2}{\tau^2 - |\xi|^2} \int_{\mathbb{S}^2} \alpha^2 \hat{\tilde{f}}\left( \frac{\alpha}{2} \omega \right) \hat{\tilde{f}} \left( \xi - \frac{\alpha}{2} \omega \right) \d^2 \omega \]
	for $\tau \geq \xi$, and zero otherwise, where
	\[ \alpha = \alpha(\xi, \tau, \omega) = \frac{\tau^2 - |\xi|^2}{\tau - \xi \cdot \omega}. \]
	\end{lemma}
	 
	 \noindent A key point is that the proof of \Cref{lemma:spacetime_convolution_lemma} relies on the inverse convolution formula \eqref{inverse_convolution_flat_space} on $\mathbb{R}^4$.

\subsection{Spacetime Fourier transform of $\widetilde{Q}_0^{\pm, \pm}$ and change of variables} Using the formula from \Cref{lemma:spacetime_convolution_lemma}, one then has for $\tau \geq |\xi|$
	\[ \mathcal{F}_{\tilde{t} \to \tau, \, \tilde{x} \to \xi}(\widetilde{Q}_0^{+,+}) = \pi^2 \int_{\mathbb{S}^2} \alpha^2 \hat{\tilde{f}} \left(\frac{\alpha}{2} \omega \right) \hat{\tilde{f}} \left( \xi - \frac{\alpha}{2} \omega \right) \d^2 \omega, \]
	so by Cauchy--Schwarz
	\[ |\mathcal{F}(\widetilde{Q}_0^{+,+})|^2 \leq \pi^2 |\mathbb{S}^2| \int_{\mathbb{S}^2} \alpha^4 \left|\hat{\tilde{f}} \left(\frac{\alpha}{2} \omega \right) \right|^2 \left| \hat{\tilde{f}} \left( \xi - \frac{\alpha}{2} \omega \right) \right|^2 \d^2 \omega. \]
	One can show, from the definition of $\alpha$, that $\frac{\d \alpha}{\d \tau} \geq 1$, so changing variables $\d \tau = \frac{\d \tau}{\d \alpha} \d \alpha$, Plancherel's theorem gives
	\begin{align*}
		\| \widetilde{Q}_0^{+,+} \|^2_{L^2(\mathbb{R}^4_{\tilde{t},\tilde{x}})} & \la \| \mathcal{F}(\widetilde{Q}_0^{+,+}) \|^2_{L^2(\mathbb{R}^4_{\tau, \xi})} \\
		& \la \int_{\mathbb{R}^3} \d \xi \int_0^\infty \d \tau \int_{\mathbb{S}^2} \d^2 \omega \, \alpha^4 	\left|\hat{\tilde{f}}\left(\frac{\alpha}{2} \omega \right) \right|^2 \left| \hat{\tilde{f}} \left( \xi - \frac{\alpha}{2} \omega \right) \right|^2 \\
		& \la \int_{\mathbb{R}^3} \d \xi \int_0^\infty \d \alpha \int_{\mathbb{S}^2} \d^2 \omega \, \alpha^4 \left|\hat{\tilde{f}}\left(\frac{\alpha}{2} \omega \right) \right|^2 \left| \hat{\tilde{f}} \left( \xi - \frac{\alpha}{2} \omega \right) \right|^2.
	\end{align*}
	Finally, reinstating the variable $\xi ' = \frac{\alpha}{2} \omega$, $\d \xi' = \alpha^2 \, \d \alpha \, \d^2 \omega$ and using Plancherel again, one arrives at
	\begin{align*}
		\| \widetilde{Q}_0^{+,+} \|^2_{L^2(\mathbb{R}^4)} & \la \int_{\mathbb{R}^3} \d \xi \int_{\mathbb{R}^3} \d \xi' \, |\xi'|^2 |\hat{\tilde{f}}(\xi - \xi')|^2 |\hat{\tilde{f}}(\xi')|^2 \\
		& \la \| \tilde{f} \|^2_{L^2(\mathbb{R}^3)} \| \boldsymbol{\nabla} \tilde{f} \|^2_{L^2(\mathbb{R}^3)}.	
	\end{align*}
	An important point in this step is that the reinstating of the variable $\xi'$ involves mixing the time and space Fourier variables since $\alpha = \alpha(\xi, \tau, \omega)$. The parts $\widetilde{Q}_0^{+,-}$ and $\widetilde{Q}_0^{-,-}$ may be estimated similarly, and one concludes that
	\[ \| \widetilde{Q}_0(\tilde{\phi}, \tilde{\phi})\|_{L^2(\mathbb{R} \times \mathbb{R}^3)} \la \| \tilde{f} \|_{L^2(\mathbb{R}^3)} \| \boldsymbol{\nabla} \tilde{f} \|_{L^2(\mathbb{R}^3)}. \]

\section{Setup on $\mathbb{R} \times \mathbb{S}^3$} \label{sec:setup_on_cylinder}

On $\mathbb{R} \times \mathbb{S}^3$, the spatial Fourier variables will be discrete since $\mathbb{S}^3$ is compact, however the temporal Fourier variable is in general continuous (cf. \Cref{sec:harmonicanalygroups} for our setup of Fourier transforms on Lie groups). To deal with this we use the time periodicity of \eqref{free_equations_with_time_data_2}.

\subsection{Time periodicity}

Solutions to \eqref{free_equations_with_time_data_2} are periodic in time; this may be seen at the level of twistor formulae (see \S9.4 of \cite{spinorsandspacetime2}), or as follows. Suppose a $C^\infty(\mathbb{R} \times \mathbb{S}^3)$ solution to \eqref{free_equations_with_time_data_2} is given. Taking the operator-valued Fourier transform in the $\mathrm{SU}(2)$ variables, one has
\[ \partial_t^2 \hat{\phi}(\pi_m) + m (m+2) \hat{\phi}(\pi_m) + \hat{\phi}(\pi_m) = 0, \]
where $m(m+2) +1 = (m+1)^2$ is a perfect square. This has the general solution
\begin{equation} \label{general_free_mode_cylinder} \hat{\phi}(\pi_m)(t) = A_m \sin ( (m+1) t) + B_m \cos ( (m+1)t ) \end{equation}
for constant $(m+1) \times (m+1)$ matrices $A_m$ and $B_m$, and in particular for every $ m \in \mathbb{N}_0$ the mode $\hat{\phi}(\pi_m)(t)$ is $2 \pi$-periodic in $t$. By the Peter--Weyl theorem, the full solution $\phi$ can be recovered from the Fourier series
\[ \phi(t,x) = \sum_{m \geq 0} (m+1) \operatorname{Tr}( \hat{\phi}(\pi_m)(t) \pi_m(x) ), \]
where $(t,x) \in \mathbb{R} \times \mathrm{SU}(2)$, and is also $2\pi$-periodic in $t$. It follows that solutions to \eqref{free_equations_with_time_data_2} can be identified with a subspace of the space of real-valued functions on $\mathbb{S}^1 \times \mathbb{S}^3 \simeq \mathrm{U}(1) \times \mathrm{SU}(2)$.

\subsection{Positive/negative frequency splitting} In the case of data $(\phi, \partial_t \phi)|_{t=0} = (0,f)$, the solution \eqref{general_free_mode_cylinder} becomes
\[ \hat{\phi}(\pi_m) = \frac{\sin( (m+1)t )}{(m+1)} \hat{f}(\pi_m). \]
In contrast to the continuous case, the denominator here never vanishes, so there is no need to consider a subspace of data $\hat{f}$ which vanishes around the origin. Straight away we then have
\[ \phi(t,x) = \sum_{m \geq 0} (m+1) \operatorname{Tr}( \hat{\phi}(\pi_m)(t) \pi_m(x) )  = \sum_{m \geq 0} \sin ((m+1)t) \operatorname{Tr}(\hat{f}(\pi_m) \pi_m(x) ). \]
Analogously to before, we define the positive and negative frequency parts of $\phi$ by
\begin{equation} \label{discreteposnegfrequencies} \phi^\pm(t,x) = \sum_{m \geq 0} e^{\pm i (m+1)t} \operatorname{Tr}(\hat{f}(\pi_m) \pi_m(x) ). \end{equation}
These solve $\Box \phi^\pm + \phi^\pm = 0$ and determine the full solution via $\phi = \frac{1}{2i}(\phi^+ - \phi^-)$, so we restrict our analysis to $\phi^\pm$.

\subsection{Spacetime Fourier transform of \texorpdfstring{$\phi^\pm$}{phi+}}

To write $\phi^\pm$ as a spacetime Fourier series, we first use the periodicity of $\phi^\pm(t,x)$ in $t$. These can then be written as standard Fourier series on $\mathbb{R}/2 \pi \mathbb{R} \simeq \mathbb{S}^1 \simeq \mathrm{U}(1)$,
\[ \phi^\pm(t,x) = \sum_{n \in \mathbb{Z}} e^{ int} \phi^{\pm}_n(x), \]
where, by comparing to \eqref{discreteposnegfrequencies},
\begin{equation} \label{S1Fourierphi} \phi^{\pm}_n(x) = \begin{cases} \operatorname{Tr}(\hat{f}(\pi_{\pm n-1}) \pi_{\pm n-1}(x)) \qquad \text{for } \pm n \geq 1, \\ \hspace{50pt} 0 \hspace{75pt} \text{otherwise}. \end{cases} \end{equation}
Hence the spacetime Fourier coefficient $\mathcal{F}(\phi^\pm)$ of $\phi^\pm$ is
\[ \mathcal{F}(\phi^\pm)(\pi_m)_n = \frac{1}{(m+1)} \hat{f}(\pi_m) \delta_{m, \pm n-1}, \]
where $m \geq 0$ is the $\mathrm{SU}(2)$ index, and $n$ is the $\mathrm{U}(1)$ index.

\subsection{Spacetime Fourier transform of \texorpdfstring{$Q_0^{\pm,\pm}$}{Q_0}}

Given a non-abelian Lie group $\mathrm{G}$ and two functions $f, \, g : \mathrm{G} \to \mathbb{R}$, their convolution is defined by
\[ (f * g)(x) = \int_{\mathrm{G}} f(y) g(x y^{-1}) \, \d \mu(y), \]
from which it follows that
\[ \widehat{(f*g)}(\pi) = \hat{f}(\pi) \hat{g}(\pi). \]
This is the \emph{forward} convolution theorem. However, since $\hat{f}(\pi)$ and $\hat{g}(\pi)$ are operators, there is in general no natural way of defining the convolution $\hat{f}(\pi) * \hat{g}(\pi)$, as there is no group structure on the unitary dual $\hat{\mathrm{G}}$. Instead, we proceed in a different way to obtain the spacetime Fourier transform of $Q_0^{\pm,\pm}$, by using the commutativity of the $\mathrm{U}(1)$ factor of $\mathrm{U}(1) \times \mathrm{SU}(2)$. The first step is to express $\mathcal{F}(Q_0^{\pm,\pm})$ in terms of $\mathcal{F}(\phi^\pm \psi^\pm)$.

\begin{lemma} The spacetime Fourier coefficients $\mathcal{F}(Q_0^{\pm,\pm})(\pi_m)_n$ of $Q_0^{\pm,\pm} = Q_0(\phi^\pm, \psi^\pm)$ are given by
\[ \mathcal{F}(Q_0^{\pm,\pm})(\pi_m)_n = \frac{1}{2} ( 1 + (m+1)^2 - n^2 ) \mathcal{F}(\phi^\pm \psi^\pm)(\pi_m)_n. \]
\end{lemma}

\begin{proof} This follows from the identity
\[ 2 Q_0(\phi^\pm, \psi^\pm) = (2 + \Box )\phi^\pm \psi^\pm  \]
and the fact that $\Box$ is formally self-adjoint on $\mathbb{S}^1 \times \mathbb{S}^3$. Taking the spacetime Fourier transform of this, we obtain
\begin{align*}
	\mathcal{F}(Q_0^{\pm,\pm})(\pi_m)_n & = \mathcal{F}(\phi^\pm \psi^\pm)(\pi_m)_n + \frac{1}{4\pi} \int_{-\pi}^{\pi} \int_{\mathrm{SU}(2)} \e^{-int} \pi_m(x^{-1}) (\Box(\phi^\pm \psi^\pm))(t,x) \, \d \mu \, \d t \\
	& = \mathcal{F}(\phi^\pm \psi^\pm)(\pi_m)_n + \frac{1}{4\pi} \int_{-\pi}^{\pi} \int_{\mathrm{SU}(2)}(\phi^\pm \psi^\pm)(t,x) \Box( \e^{-int} \pi_m(x^{-1}) ) \, \d \mu \, \d t \\
	& = \frac{1}{2} ( 1+ (m+1)^2 - n^2) \mathcal{F}(\phi^\pm \psi^\pm)(\pi_m)_n.
\end{align*}
\end{proof}

While an inverse convolution theorem on the whole of $\mathrm{U}(1) \times \mathrm{SU}(2)$ is not available, it is still available on the abelian $\mathrm{U}(1)$ factor. The $\mathrm{U}(1)$-Fourier coefficients of $(\phi^\pm \psi^\pm)$ are
\[ (\phi^\pm \psi^\pm)_n = \sum_{l \in \mathbb{Z}} \phi^\pm_l \psi^\pm_{n-l}, \]
so that, using \eqref{S1Fourierphi},
\begin{equation} \label{S1Fourierproduct++} (\phi^+ \psi^+)_n(x) = \sum_{l  = 1}^{n-1} \operatorname{Tr}(\hat{f}(\pi_{l-1}) \pi_{l-1}(x) ) \operatorname{Tr}(\hat{g}(\pi_{n-l-1}) \pi_{n-l-1}(x) )
\end{equation}
for the $(+,+)$ product and
\begin{equation} \label{S1Fourierproduct+-} 
	(\phi^+ \psi^-)_n(x) = \sum_{l = n+1}^\infty \operatorname{Tr}(\hat{f}(\pi_{l-1}) \pi_{l-1}(x)) \operatorname{Tr}(\hat{g}(\pi_{l-n-1}) \pi_{l-n-1} (x)) 
\end{equation}
for the $(+,-)$ product. The expressions for the $(-,+)$ and $(-,-)$ products are similar. Integrating \eqref{S1Fourierproduct++}, \eqref{S1Fourierproduct+-} against $\pi_m(x^{-1}) \, \d \mu(x)$, we have
\[ \mathcal{F}(\phi^+ \psi^+)(\pi_m)_n = \sum_{l=1}^{n-1} \varpi_l^{+,+}(\pi_m)_n ~~\quad \text{and} ~~\quad \mathcal{F}(\phi^+ \psi^-)(\pi_m)_n = \sum_{l = n+1}^\infty \varpi^{+,-}_l(\pi_m)_n, \]
where
\begin{align*} \varpi^{+,+}_l (\pi_m)_n & \overset{\text{def}}{=} \int_{\mathrm{SU}(2)} \operatorname{Tr}(\hat{f}(\pi_{l-1}) \pi_{l-1}(x) ) \operatorname{Tr}(\hat{g}(\pi_{n-l-1}) \pi_{n-l-1}(x) ) \pi_m (x^{-1} ) \, \d \mu (x) \\
& = \hat{f}(\pi_{l-1})_{ji} \hat{g}(\pi_{n-l-1})_{sr} \int_{\mathrm{SU}(2)} (\pi_{l-1})_{ij} (\pi_{n-l-1})_{rs} \pi_m^\dagger \, \d \mu.
\end{align*}
and
\[ \varpi^{+,-}_l(\pi_m)_n = \hat{f}(\pi_{l-1})_{ji} \hat{g}(\pi_{l-n-1})_{sr} \int_{\mathrm{SU}(2)} (\pi_{l-1})_{ij} (\pi_{l-n-1})_{rs} \pi_m^\dagger \, \d \mu. \]
Computing $\varpi_l^{+,\pm}(\pi_m)_n$ encapsulates the key novelty of this calculation as compared to the $\mathbb{R}^{1+3}$ case: as $\mathrm{SU}(2)$ is non-abelian, the tensor product $\pi_l \otimes \pi_{n-l}$ is not isomorphic to a single irreducible representation $\pi_n$. Instead, for $\mathrm{SU}(2)$ the product $\pi_a \otimes \pi_b$ is isomorphic to the Clebsch--Gordan expansion
\begin{equation} \label{abstractCG} \pi_a \otimes \pi_b = \bigoplus_{k=0}^{\min\{a,b \}} \pi_{|a-b| + 2k }
\end{equation}
(see, for example, \S7 of \cite{Faraut2008}). For calculation we will need the expansion \eqref{abstractCG} with an explicit choice of the representations $\pi_m$; we use Wigner's D-matrices \cite{Wigner1959}.

\subsection{Wigner's D-matrices and Clebsch--Gordan expansions}

Wigner's D-matrices are explicit unitary matrix representations of $\pi_m$, $m \geq 0$. The rows and columns of $(\pi_m)_{ij}$ are labelled by the integers and half-integers
\[ i, j = - \frac{m}{2}, - \frac{m}{2} + 1, \dots , \frac{m}{2} - 1, \frac{m}{2}, \]
and a general expression for all $(\pi_m)_{ij}(\alpha, \beta, \gamma)$, where $(\alpha, \beta, \gamma)$ are the Euler angles parametrizing $\mathrm{SU}(2)$, can be found on p. 167 of \cite{Wigner1959}. As an example, the first three representations are given by
\begingroup
\renewcommand{\arraystretch}{1.8}
\[ \pi_0 = 1, \qquad  \pi_1 = \pm \left( \begin{array}{lr} \e^{-\frac{1}{2} i \alpha} \left( \cos \frac{1}{2} \beta \right) \e^{-\frac{1}{2} i \gamma} & - \e^{-\frac{1}{2} i \alpha} \left( \sin \frac{1}{2} \beta \right) \e^{\frac{1}{2} i \gamma} \\ \e^{\frac{1}{2} i \alpha} \left( \sin \frac{1}{2} \beta \right) \e^{- \frac{1}{2} i \gamma} & \e^{\frac{1}{2} i \alpha} \left( \cos \frac{1}{2} \beta \right) \e^{\frac{1}{2} i \gamma}  \end{array} \right), \]
and
\[ \pi_2 = \left( \begin{array}{lcr} \e^{-i\alpha} \frac{1}{2}(1+\cos\beta) \e^{- i \gamma} & - \e^{-i\alpha} \frac{1}{\sqrt{2}} \sin \beta & \e^{-i \alpha} \frac{1}{2}(1- \cos \beta) \e^{i \gamma} \\ \frac{1}{\sqrt{2}} (\sin \beta ) \e^{-i \gamma} & \cos \beta & -\frac{1}{\sqrt{2}} ( \sin \beta) \e^{i \gamma} \\ \e^{i \alpha} \frac{1}{2} ( 1 - \cos \beta) \e^{-i\gamma} & \e^{i\alpha} \frac{1}{\sqrt{2}} \sin \beta & \e^{i\alpha} \frac{1}{2} ( 1+ \cos \beta) \e^{i\gamma} \end{array} \right). \]
\endgroup
Here $\alpha, \beta, \gamma \in [-\pi, \pi]$, and the rows and columns of $\pi_1$ are labelled by $-\frac{1}{2}, \frac{1}{2}$, and of $\pi_2$ by $-1, 0, 1$. In terms of Wigner's D-matrices, an indexed version of the Clebsch--Gordan expansion \eqref{abstractCG} is given by (see eq. (17.16b) in \cite{Wigner1959}, or, for example, eq. (4.25) in \cite{Rose1957})
\begin{align} \begin{split} \label{indexCG} (\pi_a)_{ij}(\pi_b)_{rs} &= \sum_{k = 0}^{\min\{a,b\}} \left\langle\left. \frac{a}{2} \, i \, \frac{b}{2} \, r \, \right| \frac{|a-b| + 2k}{2} \, (i+r) \right\rangle \\
& \hspace{36pt} \times \left\langle\left. \frac{a}{2} \, j \, \frac{b}{2} \, s \, \right| \frac{|a-b| + 2k}{2} \, (j+s) \right\rangle (\pi_{|a-b| + 2k})_{(i+r)(j+s)}, 
\end{split}	
\end{align}
where $\langle j_1 \, m_1 \, j_2 \, m_2 | j_3 \, m_3 \rangle $ are the standard Clebsch--Gordan coefficients of $\mathrm{SU}(2)$. The Clebsch--Gordan coefficients are probability amplitudes, so in particular they are uniformly bounded by 1 in all six variables $j_1, j_2, j_3, m_1, m_2, m_3$. They do not vanish only if $-j_1 \leq m_1 \leq j_1$, $-j_2 \leq m_2 \leq j_2$, $-j_3 \leq m_3 \leq j_3$, and $m_1 + m_2 = m_3$, and further obey the orthogonality relations (cf. eqs. (3.7), (3.9) in \cite{Rose1957})
\begin{align}
	\label{orth1} & \sum_{j_3} \sum_{m_3} \langle j_1 \, m_1 \, j_2 \, m_2 | j_3 \, m_3 \rangle \langle j_1 \, m_1' \, j_2 \, m_2' | j_3 \, m_3 \rangle = \delta_{m_1, m_1'} \delta_{m_2, m_2'}, \\
	\label{orth2} & \sum_{m_1} \sum_{m_2} \langle j_1 \, m_1 \, j_2 \, m_2 | j_3 \, m_3 \rangle \langle j_1 \, m_1 \, j_2 \, m_2 | j_3' \, m_3' \rangle = \delta_{j_3, j_3'} \delta_{m_3, m_3'},
\end{align}
where in \eqref{orth1} the sums run over $|j_1 - j_2| \leq j_3 \leq j_1 + j_2$ and $-j_3 \leq m_3 \leq j_3$, and in \eqref{orth2} the sums run over $-j_{1,2} \leq m_{1,2} \leq j_{1,2}$. Using \eqref{indexCG} and the Schur orthogonality relations, we therefore have
\begin{align*} & \int_{\mathrm{SU}(2)} (\pi_{l-1})_{ij} ( \pi_{n-l-1})_{rs} (\pi^\dagger_m)_{pq} \, \d \mu \\
 & = \sum_{k=0}^{\min\{l-1, n-l-1\}} \left\langle \left.\frac{l-1}{2} \, i \, \frac{n-l-1}{2} \, r \right| \frac{|n-2l| + 2k}{2} \, (i+r) \right\rangle \\
 & \hspace{40pt} \times \left\langle \left.\frac{l-1}{2} \, j \, \frac{n-l-1}{2} \, s \right| \frac{|n-2l| + 2k}{2} \, (j+s) \right\rangle \int_{\mathrm{SU}(2)} ( \pi_{|n-2l| + 2k })_{(i+r)(j+s)} (\pi^\dagger_m)_{pq} \, \d \mu \\
 & = \sum_{k=0}^{\min\{l-1, n-l-1\}} \left\langle \left.\frac{l-1}{2} \, i \, \frac{n-l-1}{2} \, r \right| \frac{|n-2l| + 2k}{2} \, (i+r) \right\rangle \\
 & \hspace{40pt} \times \left\langle \left.\frac{l-1}{2} \, j \, \frac{n-l-1}{2} \, s \right| \frac{|n-2l| + 2k}{2} \, (j+s) \right\rangle \frac{1}{(m+1)} \delta_{i+r, p} \delta_{j+s, q} \delta_{|n-2l| + 2k, m}
\end{align*}
provided
\[ |n-2l| \leq m \leq |n-2l| + 2 \min\{ l - 1, n-l-1 \} \iff |n-2l| \leq m \leq n - 2 \]
so that $\pi_m$ is in the range of the Clebsch--Gordan expansion \eqref{indexCG}, and
\[ \frac{m - |n-2l|}{2} \in \mathbb{Z} \iff m \equiv n \mod 2. \]
We therefore obtain
\begin{align} \begin{split} \label{SU2convolution} (\varpi^{+,+}_l(\pi_m)_n)_{pq} &= \frac{1}{(m+1)} \hat{f}(\pi_{l-1})_{ji} \hat{g}(\pi_{n-l-1})_{(q-j)(p-i)} \\
& \times \left\langle\left. \frac{l-1}{2} \, i \, \frac{n-l-1}{2} \,  (p-i) \right| \frac{m}{2} \, p \right\rangle \left\langle\left. \frac{l-1}{2} \, j \, \frac{n-l-1}{2} \, (q-j) \right| \frac{m}{2} \, q \right\rangle \chi_1 \chi_2,
\end{split}
\end{align}
where $\chi_1 = \chi(|n-2l| \leq m \leq n -2 )$ and $\chi_2 = \chi(m \equiv n \mod 2)$, where $\chi$ is the indicator function. Similarly, we obtain
\begin{align}
	\begin{split}
	\label{varpi_plus_minus}
		(\varpi^{+,-}_l(\pi_m)_n)_{pq} & = \frac{1}{(m+1)} \hat{f}(\pi_{l-1})_{ji} \hat{g}(\pi_{l-n-1})_{(q-j)(p-i)} \\
		& \times \left\langle \left. \frac{l-1}{2} \, i \, \frac{l-n-1}{2} \, (p-i) \, \right| \, \frac{m}{2} \, p \right\rangle \left\langle \left. \frac{l-1}{2} \, j \, \frac{l-n-1}{2} \, (q-j) \, \right| \, \frac{m}{2} \, q \right\rangle \chi_2 \chi_3,
	\end{split}
\end{align}
where $\chi_3 = \chi(|n| \leq m \leq -n +2(l-1))$, and where in both cases for a well-defined representation $\pi_{l-1}$ we require $l \geq 1$. The presence of $\chi_1$ and $\chi_3$ encodes the multi-dimensional nature of the irreducible representations of $\mathrm{SU}(2)$ and the fact that the product $\pi_l \otimes \pi_{n-l}$ has a non-trivial Clebsch--Gordan expansion. For the convolutions $\varpi^{+,\pm}_l(\pi_m)_n$ we have the following discrete Fourier-side version of Young's inequality.

\begin{lemma} \label{lemma:discrete_convolution_bound_CG_orthogonality} For the matrices $\varpi^{+,\pm}_l(\pi_m)_n$ given by \eqref{SU2convolution} and \eqref{varpi_plus_minus} we have the bounds
\[ \sum_{m \geq 0} \vertiii{(m+1) \varpi^{+,\pm}_l(\pi_m)_n}^2 \leq \vertiii{\hat{f}(\pi_{l-1})}^2 \vertiii{\hat{g}(\pi_{\pm(n-l)-1})}^2. \]
\end{lemma}

\begin{proof} We write down the proof for $\varpi^{+,+}_l(\pi_m)_n$, the calculation for $\varpi^{+,-}_l(\pi_m)_n$ being the same. For more efficient notation, we write here
\[ \sum_{m \geq 0} \chi_1 \chi_2 = \sum_m \]
 and 
\[ \langle j_1 \, m_1 \, j_2 \, m_2 \, | \, j_3 \, m_3 \rangle = \mathcal{C}^{(j_1 j_2)}_{j_3 m_1 m_2}, \]
with the understanding that the ranges of the sums below are restricted to the ranges in which the Clebsch--Gordan coefficients involved are non-zero. Then the orthogonality relations \eqref{orth1} and \eqref{orth2} read
\begin{align}
	\label{orth1'} & \sum_{j_3} \sum_{m_3} \mathcal{C}^{(j_1 j_2)}_{j_3 m_1 m_2} \mathcal{C}^{(j_1 j_2)}_{j_3 m'_1 m_2'} = \delta_{m_1, m_1'} \delta_{m_2, m_2'}, \\
	\label{orth2'} & \sum_{m_1} \sum_{m_2} \mathcal{C}^{(j_1 j_2)}_{j_3 m_1 m_2}  \mathcal{C}^{(j_1 j_2)}_{j'_3 m_1 m_2} = \delta_{j_3, j_3'} .
\end{align}
Explicitly expanding the Frobenius norm squared of $(m+1)\varpi^{+,+}_l(\pi_m)_n$, we have 
\begin{align*}
	\sum_{m \geq 0} & \vertiii{(m+1) \varpi^{+,+}_l(\pi_m)_n}^2 = \sum_m \sum_{p, q} \sum_{i,j} \sum_{i',j'} \hat{f}(\pi_{l-1})_{ji} \overline{\hat{f}(\pi_{l-1})}_{j'i'}  \\
	& \times \hat{g}(\pi_{n-l-1})_{(q-j)(p-i)} \overline{\hat{g}(\pi_{n-l-1})}_{(q-j')(p-i')} \mathcal{C}^{\left( \frac{l-1}{2} \frac{n-l-1}{2} \right)}_{\frac{m}{2} i (p-i)} \mathcal{C}^{\left( \frac{l-1}{2} \frac{n-l-1}{2} \right)}_{\frac{m}{2} i' (p-i')} \mathcal{C}^{\left( \frac{l-1}{2} \frac{n-l-1}{2} \right)}_{\frac{m}{2} j (q-j)} \mathcal{C}^{\left( \frac{l-1}{2} \frac{n-l-1}{2} \right)}_{\frac{m}{2} j' (q-j')} \\
	& = \sum_m \sum_{p,q} \left| \sum_{i,j} \mathcal{C}^{\left( \frac{l-1}{2} \frac{n-l-1}{2} \right)}_{\frac{m}{2} i (p-i)} \mathcal{C}^{\left( \frac{l-1}{2} \frac{n-l-1}{2} \right)}_{\frac{m}{2} j (q-j)} \hat{f}(\pi_{l-1})_{ji} \hat{g}(\pi_{n-l-1})_{(q-j)(p-i)} \right|^2 \\
	& = \sum_m \sum_{r,s} \left| \sum_{i,j} \mathcal{C}^{\left( \frac{l-1}{2} \frac{n-l-1}{2} \right)}_{\frac{m}{2} i r} \mathcal{C}^{\left( \frac{l-1}{2} \frac{n-l-1}{2} \right)}_{\frac{m}{2} j s} \hat{f}(\pi_{l-1})_{ji} \hat{g}(\pi_{n-l-1})_{sr} \right|^2 \\
	& \leq \vertiii{\hat{g}(\pi_{n-l-1})}^2 \sum_m \sum_{r,s} \left| \sum_{i,j} \mathcal{C}^{\left( \frac{l-1}{2} \frac{n-l-1}{2} \right)}_{\frac{m}{2} i r} \mathcal{C}^{\left( \frac{l-1}{2} \frac{n-l-1}{2} \right)}_{\frac{m}{2} j s} \hat{f}(\pi_{l-1})_{ji} \right|^2 \\
	& \leq \vertiii{\hat{g}(\pi_{n-l-1})}^2 \sum_m \sum_{m'} \sum_{r,s} \left| \sum_{i,j} \mathcal{C}^{\left( \frac{l-1}{2} \frac{n-l-1}{2} \right)}_{\frac{m}{2} i r} \mathcal{C}^{\left( \frac{l-1}{2} \frac{n-l-1}{2} \right)}_{\frac{m'}{2} j s} \hat{f}(\pi_{l-1})_{ji} \right|^2 \\
	& = \vertiii{\hat{g}(\pi_{n-l-1})}^2 \sum_m \sum_r \sum_{i,i'} \bigg[ \mathcal{C}^{\left( \frac{l-1}{2} \frac{n-l-1}{2} \right)}_{\frac{m}{2} i r} \mathcal{C}^{\left( \frac{l-1}{2} \frac{n-l-1}{2} \right)}_{\frac{m}{2} i' r} \\
	& \hspace{68pt} \times \sum_{m'} \sum_s \sum_{j,j'} \mathcal{C}^{\left( \frac{l-1}{2} \frac{n-l-1}{2} \right)}_{\frac{m'}{2} j s}  \mathcal{C}^{\left( \frac{l-1}{2} \frac{n-l-1}{2} \right)}_{\frac{m'}{2} j' s} \hat{f}(\pi_{l-1})_{ji} \overline{\hat{f}(\pi_{l-1})}_{j'i'} \bigg] \\
	& = \vertiii{\hat{f}(\pi_{l-1})}^2 \vertiii{\hat{g}(\pi_{n-l-1})}^2,
\end{align*}
where in the last line we used \eqref{orth1'} twice.
\end{proof}

\section{Proof of the Basic Estimate} \label{sec:basic_estimate_on_cylinder}

We now turn to the proof of the basic estimate of \Cref{thm:basic_estimate}, i.e.
\[ \| Q_0(\phi, \psi) \|^2_{L^2([-\pi,\pi] \times \mathbb{S}^3)} \la \| f \|^2_{H^{1}(\mathbb{S}^3)} \| g \|^2_{H^\epsilon(\mathbb{S}^3)}. \]
Noting that $-4Q_0(\phi, \psi) = Q_0^{+,+} - Q_0^{+,-} - Q_0^{-,+} + Q_0^{-,-}$, it is enough to estimate each of the four parts $Q_0^{\pm,\pm}$. Clearly the estimates for $Q_0^{+,+}$ and $Q_0^{-,-}$ are entirely analogous, as are those for $Q_0^{+,-}$ and $Q_0^{-,+}$. It is therefore enough to show that $Q_0^{+,+}$ and $Q_0^{+,-}$ are suitably controlled. We record the following auxiliary lemma.

\begin{lemma} \label{lem:behaviour_lambda} For any fixed $n \in \mathbb{Z}$ the function $\lambda: \mathbb{R}_{\geq 0} \to \mathbb{R}$ defined by
\[ \lambda(m; n) \overset{\text{def}}{=} \frac{(1+(m+1)^2 - n^2)^2}{(m+1)} \]
is decreasing for $m \in [0, \sqrt{n^2-1}-1)$, has a global minimum $\lambda(m_*) = 0$ at $m_* = \sqrt{n^2-1} -1$, and is increasing for $m \in  (\sqrt{n^2-1}-1, \infty)$.
\end{lemma}
\begin{proof}
This follows directly from computing the derivative in $m$,
\[ \lambda'(m; n) = \frac{(3(m+1)^2 + (n^2 - 1))}{(m+1)^{2}} ( 1 + (m+1)^2 - n^2) , \]
where the first factor in the above is manifestly positive for all $n \in \mathbb{Z}$ and $m \geq 0$. 	
\end{proof}

\subsection{Estimate for $Q_0^{+,+}$}

By Plancherel's theorem, we compute, noting that we must have $n \geq 2$ due to $\chi_1 = \chi(|n-2l| \leq m \leq n -2)$,

\begin{align}
\begin{split}
\label{Q0++estimate}
 \| Q_0^{+,+} \|^2_{L^2(\mathbb{S}^1 \times \mathbb{S}^3)} & \simeq \sum_{n \geq 2} \sum_{m \geq 0} (m+1) \vertiii{\mathcal{F}(Q_0^{+,+})(\pi_m)_n}^2 \\\
& \simeq \sum_{n \geq 2} \sum_{m \geq 0} (1+(m+1)^2 - n^2 )^2 ( m+ 1) \vertiii{\mathcal{F}(\phi^+ \psi^+)(\pi_m)_n}^2 \\
& \simeq \sum_{n \geq 2} \sum_{m \geq 0} ( 1 + (m+1)^2 - n^2 )^2 ( m + 1)^{-1} \vertiii{\sum_{l=1}^{n-1} (m+1) \varpi^{+,+}_l(\pi_m)_n}^2
\end{split}
\end{align}
We now perform the splitting, for any fixed $\epsilon > 0$,
\begin{align*} \vertiii{\sum_{l=1}^{n-1} (m+1) \varpi^{+,+}_l(\pi_m)_n }^2 &\leq \sum_{l=1}^{n-1} \frac{(1+|n-2l|) l}{(n-l)^{1-2\epsilon}} \vertiii{(m+1) \varpi^{+,+}_l(\pi_m)_n}^2 \\
& \times \sum_{k=1}^{n-1} \frac{(n-k)^{1-2\epsilon}}{(1+|n-2k|) k},
\end{align*}
where the last factor in the above converges as $n \to \infty$ for $\epsilon > 0$ by \Cref{lem:aux_2} with $\beta_0 = 0$, $\alpha = 1$, $\beta = -1 + 2\epsilon$. Now note that in estimate \eqref{Q0++estimate} the function $\lambda(m) = (1+ (m+1)^2 - n^2)^2 (m+1)^{-1}$ is decreasing in $m$ in the range $|n-2l| \leq m \leq n-2$ (cf. \Cref{lem:behaviour_lambda}), so its maximum in this range is attained at $m=|n-2l|$. Using \Cref{lemma:discrete_convolution_bound_CG_orthogonality}, we have
\begin{align}
	\label{Q0++estimate_continued}
	\begin{split}
	\| Q_0^{+,+} \|^2_{L^2(\mathbb{S}^1 \times \mathbb{S}^3)} & \la_\epsilon \sum_{n \geq 2} \sum_{m \geq 0} \sum_{l=1}^{n-1} (1+(m+1)^2 - n^2)^2 (m+1)^{-1} \frac{(1+|n-2l|) l}{(n-l)^{1-2\epsilon}} \\
	& \times \vertiii{(m+1) \varpi^{+,+}_l(\pi_m)_n}^2 \\
	&\la_\epsilon \sum_{n \geq 2} \sum_{l=1}^{n-1} \frac{(1+(|n-2l|+1)^2 - n^2)^2}{(n-l)^{1-2\epsilon}} l \vertiii{\hat{f}(\pi_{l-1})}^2 \vertiii{\hat{g}(\pi_{n-l-1})}^2 \\
n-l \overset{\text{def}}{=} k \longrightarrow & \la_\epsilon \sum_{l \geq 1} \sum_{k \geq 1} \frac{(1+(|k-l|+1)^2 - (k+l)^2)^2}{k^{1-2\epsilon}} l \vertiii{\hat{f}(\pi_{l-1})}^2 \vertiii{\hat{g}(\pi_{k-1})}^2 \\
	& \la_\epsilon \sum_{l \geq 1} \sum_{k \geq 1} \frac{k^2 l^2 }{k^{1-2\epsilon}} l \vertiii{\hat{f}(\pi_{l-1})}^2 \vertiii{\hat{g}(\pi_{k-1})}^2 \\
	& \la_\epsilon \| f \|^2_{H^{1}(\mathbb{S}^3)} \| g \|^2_{H^\epsilon(\mathbb{S}^3)}.
	\end{split}
\end{align}

\subsection{Estimate for $Q_0^{+,-}$}

Here we must split the estimate into the two cases $n \geq 0$ and $n < 0$. We have
\begin{align*}
	\| Q_0^{+,-} \|^2_{L^2(\mathbb{S}^1 \times \mathbb{S}^3)} & \simeq \sum_{n \in \mathbb{Z}} \sum_{m \geq 0} (m+1) \vertiii{\mathcal{F}(Q_0^{+,-})(\pi_m)_n}^2 \\
	& \simeq \bigg[\sum_{n \geq 0 } + \sum_{n < 0} \bigg] \underbrace{\sum_{m \geq 0} (1+(m+1)^2 - n^2)^2 (m+1) \vertiii{\sum_{\substack{l \geq n+1 \\ l \geq 1}} \varpi^{+,-}_l(\pi_m)_n }^2}_{(\star)} .\\
\end{align*}

\subsubsection*{Case $n<0$.}
In the case $n < 0$, we have, for any fixed $\epsilon > 0$,
\begin{align*}
	\sum_{n < 0} (\star) & = \sum_{n<0} \sum_{m \geq 0} (1+(m+1)^2-n^2)(m+1)^{-1} \vertiii{\sum_{l \geq 1} (m+1)\varpi^{+,-}_l(\pi_m)_n }^2 \\
	& \la \sum_{n < 0} \sum_{m \geq 0} (1+ (m+1)^2 - n^2)^2 (m+1)^{-1} \sum_{l \geq 1} l (l + |n|)^{2\epsilon} \vertiii{(m+1)\varpi^{+,-}_l(\pi_m)_n}^2 \\
	&\phantom{\la} \times \sum_{k \geq 1} k^{-1} (k+|n|)^{-2\epsilon}.
\end{align*}
Now in the range $|n| \leq m \leq |n| + 2(l-1)$ the function $\lambda(m;n) = (1+(m+1)^2 -n^2)^2 (m+1)^{-1}$ is increasing, so is maximised at $m = |n| + 2(l-1)$. Hence, using \Cref{lemma:discrete_convolution_bound_CG_orthogonality},
\begin{align*}
	\sum_{n < 0} (\star) & \la \sum_{n < 0 } \sum_{l \geq 1} \frac{(1+(|n|+2l -1)^2 -n^2)^2}{(|n| + 2l -1)} l (l+|n|)^{2\epsilon} \sum_m \vertiii{(m+1)\varpi^{+,-}_l(\pi_m)_n}^2 \underbrace{\sum_{k \geq 1} k^{-1-2\epsilon}}_{= \zeta(1+2\epsilon) < \infty } \\
	& \la_\epsilon \sum_{n < 0} \sum_{l \geq 1} \frac{l^2(l+|n|)^2}{(l + |n|)} l (l+|n|)^{2\epsilon} \vertiii{\hat{g}(\pi_{l-n-1})}^2 \vertiii{\hat{f}(\pi_{l-1})}^2 \\
	& \la_\epsilon \sum_{n < 0} \sum_{l \geq 1} l^{3} \vertiii{\hat{f}(\pi_{l-1})}^2 (l+|n|)^{1+2\epsilon} \vertiii{\hat{g}(\pi_{l-n-1})}^2 \\
	& \la_\epsilon  \| f \|^2_{H^{1}(\mathbb{S}^3)} \| g \|^2_{H^\epsilon(\mathbb{S}^3)} \, .
\end{align*}

\subsubsection*{Case $n \geq 0$.} Here we proceed similarly; for any fixed $\epsilon > 0$ we have
\begin{align*}
		\sum_{n \geq 0} (\star) & = \sum_{n \geq 0} \sum_{m \geq 0} (1+(m+1)^2-n^2)(m+1)^{-1} \vertiii{\sum_{l \geq n + 1} (m+1)\varpi^{+,-}_l(\pi_m)_n }^2 \\
	& \la \sum_{n \geq 0} \sum_{m \geq 0} (1+ (m+1)^2 - n^2)^2 (m+1)^{-1} \sum_{l \geq n + 1} l (l-n)^{2\epsilon} \vertiii{(m+1)\varpi^{+,-}_l(\pi_m)_n}^2 \\
	& \phantom{\la} \times \sum_{k \geq n+1} k^{-1} (k-n)^{-2\epsilon},
\end{align*}
where now the range of $m$ is $n \leq m \leq -n+2(l-1)$, so that $\lambda(m)$ is increasing and so maximised at $m = -n+2(l-1)$. Hence, using \Cref{lemma:discrete_convolution_bound_CG_orthogonality} again,
\begin{align*}
		\sum_{n \geq 0} (\star) & \la \sum_{n \geq 0} \sum_{m \geq 0} (1+ (m+1)^2 - n^2)^2 (m+1)^{-1} \sum_{l \geq n + 1} l (l-n)^{2\epsilon} \vertiii{(m+1)\varpi^{+,-}_l(\pi_m)_n}^2 \\
		&\phantom{\la} \times \underbrace{\sum_{k \geq n+1} (k-n)^{-1-2\epsilon}}_{\leq \zeta(1+2\epsilon) < \infty} \\
		& \la_\epsilon \sum_{n \geq 0} \frac{(1+(2l -n -1))^2 - n^2)^2}{(2l-n-1)} \sum_{ l \geq n +1} l (l-n)^{2\epsilon} \vertiii{\hat{g}(\pi_{l-n-1})}^2 \vertiii{\hat{f}(\pi_{l-1})}^2 \\
		& \la_\epsilon \sum_{n \geq 0} \sum_{l \geq n+1} \frac{l^2(l-n)^2}{(l-n)} l (l-n)^{2\epsilon} \vertiii{\hat{g}(\pi_{l-n-1})}^2 \vertiii{\hat{f}(\pi_{l-1})}^2 \\
		& \la_\epsilon \|f \|^2_{H^{1}(\mathbb{S}^3)} \| g \|^2_{H^\epsilon(\mathbb{S}^3)} \, .
\end{align*}
We have therefore shown that for any fixed $\epsilon > 0$
\[ \|Q_0^{+,-}\|^2_{L^2(\mathbb{S}^1 \times \mathbb{S}^3)} \la_\epsilon \| f \|^2_{H^{1}(\mathbb{S}^3)} \| g \|^2_{H^\epsilon(\mathbb{S}^3)}.  \]
Putting this together with estimate \eqref{Q0++estimate_continued}, we conclude the proof of \Cref{thm:basic_estimate}. \hfill \qed

\begin{remark}
	It is the application of H\"older's inequality to split the sum over $l$ (i.e. over the Clebsch--Gordan expansion) in the above estimates that is fundamentally responsible for the loss of $\epsilon$ amount of differentiability; it is only here that the estimate is non-uniform in $\epsilon$. The constant $C(\epsilon)$ in \Cref{cor:inhomogeneous_basic_estimate} diverges like $\zeta(1)$, i.e. logarithmically, as $\epsilon \to 0$.
\end{remark}

\begin{remark} By a similar calculation (or as a corollary of \Cref{thm:estimates_with_multipliers} below with $\beta_w = 0$, $\beta_0 = -\epsilon$, $\alpha_1 = 1+\epsilon$, $\alpha_2 = \epsilon$), one also obtains 
\begin{equation}
	\label{basic_estimate_with_extra_epsilons}
	 \| Q_0(\phi, \psi) \|^2_{L^2([-\pi,\pi];H^{\epsilon}(\mathbb{S}^3))} \la \| f \|^2_{H^{1+\epsilon}(\mathbb{S}^3)} \| g \|^2_{H^{\epsilon}(\mathbb{S}^3)}.
\end{equation}
\end{remark}

\section{Estimates with Multipliers} \label{sec:estimates_with_multipliers}

Write $J = (1 - \slashed{\Delta})^{1/2}$ and $W = (2+ \Box)$ as before. We now turn to the proof of the most general estimate of \Cref{thm:estimates_with_multipliers}, i.e.
\[ \| J^{-\beta_0} W^{\beta_w} Q_0(\phi, \psi) \|_{L^2(\mathbb{S}^1 \times \mathbb{S}^3)} \la \| f \|_{H^{\alpha_1}(\mathbb{S}^3)} \| g \|_{H^{\alpha_2}(\mathbb{S}^3)}, \]
with the requirements on $\alpha_1$, $\alpha_2$, $\beta_0$ and $\beta_w$ as stated in \eqref{alpha_1_alpha_2_condition_1}--\eqref{alpha_1_alpha_2_condition_3} (see \Cref{thm:estimates_with_multipliers}).

\begin{proof}[Proof of \Cref{thm:estimates_with_multipliers}]

We proceed as in the case of \Cref{thm:basic_estimate}, by splitting the null form into the four terms $Q_0^{+,+}$, $Q_0^{+,-}$, $Q_0^{-,+}$ and $Q_0^{-,-}$. It is enough to treat $Q_0^{+,+}$ and $Q_0^{+,-}$. Recall that the operator $J = (1-\slashed{\Delta})^{1/2}$ is the Fourier multiplier by $(m+1)$ and $W = (2+\Box)$ is the Fourier multiplier by $|1+(m+1)^2 - n^2|$. 

\subsection*{The $(+,+)$ component.} We have
\begin{align*}
	& \| J^{-\beta_0} W^{\beta_w}Q_0^{+,+} \|^2_{L^2(\mathbb{S}^1 \times \mathbb{S}^3)} \\
	& \simeq \sum_{n \geq 2} \sum_{m \geq 0} (m+1)^{1-2\beta_0} |1 + ( m+1)^2 - n^2|^{2\beta_w} \vertiii{\mathcal{F}(Q_0^{+,+})(\pi_m)_n}^2 \\
	& \simeq \sum_{n \geq 2} \sum_{m \geq 0} (m+1)^{-1 - 2\beta_0} | 1 +(m+1)^2 - n^2|^{2+2\beta_w} \vertiii{ \sum_{l=1}^{n-1}(m+1) \varpi^{+,+}_l(\pi_m)_n }^2,
\end{align*}
where we wish to split the sum in $\vertiii{\cdot}^2$ in the last term in such a way as to eliminate the $(m+1)^{-1-2\beta_0}$ term. In order to do this, we need to understand the maximum of
\[ \lambda(m; n, \beta_0, \beta_w) \overset{\text{def}}{=} (m+1)^{-1-2\beta_0} |n^2-1-(m+1)^2|^{2+2\beta_w}. \]
Note that for $|n-2l| \leq m \leq n-2$ we have $n^2 - 1 - (m+1)^2 \geq 2(n-1) \geq 2$ for $n\geq 2$, so the term in absolute value is in fact positive. Differentiating $\lambda$ gives
\begin{align*} \frac{\d \lambda}{\d m}(m ; n, \beta_0, \beta_w) &= - ((1+2\beta_0)(n^2 - 1 - (m+1)^2) + 2(m+1)^2 (2 + 2\beta_w) ) \\
& \phantom{=} \ \times (m+1)^{-2-\beta_0} (n^2 - 1 - (m+1)^2)^{1+2\beta_w},	
\end{align*}
which is always non-positive provided $\beta_0 \geq - 1/2$ and $\beta_w \geq - 1$ (in the range of the $(+,+)$ component, i.e. $|n-2l| \leq m \leq n-2$). Hence $\lambda(m;n,\beta_0, \beta_w)$ is maximized at $\lambda(|n-2l|;n,\beta_0, \beta_w)$ in this range, so we have by \Cref{lem:aux_2},
\begin{align*}
	& \| J^{-\beta_0} W^{\beta_w} Q_0^{+,+} \|^2_{L^2(\mathbb{S}^1 \times \mathbb{S}^3)} \\
	& \la S(\beta_0, \alpha, \beta; \infty) \sum_{n \geq 2} \sum_{l=1}^{n-1} |1 + (|n-2l| + 1)^2 - n^2 |^{2 + 2\beta_w} l^\alpha (n-l)^{\beta} \vertiii{\hat{f}(\pi_{l-1})}^2 \vertiii{\hat{g}(\pi_{n-l-1})}^2 \\
	& \la \sum_{k \geq 1} \sum_{l \geq 1} k^{2+2\beta_w} l^{2+2\beta_w} l^\alpha k^\beta \vertiii{\hat{f}(\pi_{l-1})}^2 \vertiii{\hat{g}(\pi_{k-1})}^2 \\
	& \la \| f \|^2_{H^{\alpha_1}(\mathbb{S}^3)} \| g \|^2_{H^{\alpha_2}(\mathbb{S}^3)},
\end{align*}
provided $\alpha$, $\beta$ and $\beta_0$ satisfy the conditions of \Cref{lem:aux_2}, and where $1 + 2\alpha_1 = 2 + 2\beta_w + \alpha$ and $1 + 2\alpha_2 = 2 + 2\beta_w + \beta$. Translating in terms of $\alpha_1$ and $\alpha_2$, one obtains the conditions in the statement of \Cref{thm:estimates_with_multipliers}.

\subsection*{The $(+,-)$ component.} The range of $m$ is now $|n| \leq m \leq - n + 2(l-1)$, and therefore $(m+1)^2 \geq n^2$, so $\lambda(m;n,\beta_0,\beta_w)$ is given by 
\[ \lambda(m;n,\beta_0,\beta_w) = (m+1)^{-1-2\beta_0} (1+ (m+1)^2 - n^2)^{2+2\beta_w}. \]
Its derivative is given by
\begin{align*}
	\frac{\d \lambda}{\d m}(m;n,\beta_0,\beta_w) &= (m+1)^{-2-2\beta_0} (1+(m+1)^2 - n^2)^{1+2\beta_w} \\
	& \times \left( (n^2 - 1)(1+2\beta_0) + (m+1)^2(3+4\beta_w - 2\beta_0) \right),	
\end{align*}
which, with the exception of the $n=0$ mode, is uniformly positive in $m$ provided
\[ 1+ 2\beta_0 \geq 0 \quad \text{and} \quad 3 + 4\beta_w - 2\beta_0 \geq 0, \]
i.e. exactly when \eqref{alpha_1_alpha_2_condition_3} is satisfied. We therefore have
\begin{align*}
	\| J^{-\beta_0} W^{\beta_w} Q_0^{+,-} \|^2_{L^2(\mathbb{S}^1 \times \mathbb{S}^3)} &\simeq \sum_{n \in \mathbb{Z}} \sum_{m \geq 0} \lambda(m;n,\beta_0,\beta_w) \vertiii{\sum_{\substack{l \geq n+1 \\ l \geq 1}} (m+1) \varpi_l^{+,-}(\pi_m)_n }^2,
\end{align*}
and we split the estimate into the cases $n \neq 0$ and $n=0$.

\subsubsection*{Case $n \neq 0$}

When $n \neq 0$, $\lambda(m;n,\beta_0,\beta_w)$ is increasing in $m$ in this range, so is maximised at $m = - n + 2(l-1)$. Hence for a fixed $l$, noting that $2l - n - 1 \ga l^{1/2} (l-n)^{1/2}$,
\begin{align*}
	\lambda(m;n,\beta_0,\beta_w) & \leq (2l - n -1)^{-1-2\beta_0} ( 1+ (2l-n-1)^2 - n^2 )^{2+2\beta_w} \\
	& \la l^{-1/2 - \beta_0} (l-n)^{-1/2 - \beta_0} l^{2+2\beta_w} (l-n)^{2+2\beta_w}.
\end{align*}
Hence, for any $\alpha \geq 0$, $\beta \geq 0$ such that $\alpha + \beta > 1$,
\begin{align*}
	\sum_{n \neq 0} \sum_{m \geq 0} & \lambda(m;n,\beta_0, \beta_w) \vertiii{\sum_{\substack{l \geq n+1 \\ l \geq 1}} (m+1) \varpi_l^{+,-}(\pi_m)_n }^2 \\
	& \la \sum_{n \neq 0} \sum_{l \geq n + 1} l^{3/2+2\beta_w-\beta_0} (l-n)^{3/2+2\beta_w - \beta_0} l^\alpha (l-n)^\beta \vertiii{\hat{f}(\pi_{l-1})}^2 \vertiii{\hat{g}(\pi_{l-n-1})}^2 \\
	& \la \| f \|^2_{H^{\alpha_1}(\mathbb{S}^3)} \|g \|^2_{H^{\alpha_2}(\mathbb{S}^3)},
\end{align*}
where $1+2\alpha_1 = 3/2 + 2\beta_w - \beta_0 + \alpha$ and $1+2\alpha_2  = 3/2 + 2\beta_w - \beta_0 + \beta$. Hence we require 
\[ \alpha_1 \geq \beta_w + \frac{1}{4} - \frac{1}{2} \beta_0 \quad \text{and} \quad \alpha_2 \geq \beta_w + \frac{1}{4} - \frac{1}{2} \beta_0, \]
which follow from \cref{alpha_1_alpha_2_condition_3,alpha_1_alpha_2_condition_4,alpha_1_alpha_2_condition_5}, and
\[ \alpha_1 + \alpha_2 + \beta_0 > 1 + 2\beta_w, \]
which is exactly \eqref{alpha_1_alpha_2_condition_1}.

\subsubsection*{Case $n=0$} The $n=0$ mode we treat separately, noting that now $ 0 \leq m \leq 2(l-1)$, and we seek to control
\begin{align*}
	\sum_{m \geq 0} (m+1)^{4\beta_w + 3 - 2\beta_0} & \vertiii{ \sum_{l \geq 1 + \frac{m}{2}} (m+1) \varpi^{+,-}_l(\pi_m)_0}^2 \\
	& \la_\theta \sum_{m \geq 0} \sum_{l \geq 1} (l-1)^{4	\beta_w + 3 - 2\beta_0 + \theta} \vertiii{ (m+1) \varpi^{+,-}_l(\pi_m)_0}^2 \\
	& \la_\theta \sum_{l \geq 1} (l-1)^{4\beta_w + 3 - 2\beta_0 + \theta} \vertiii{\hat{f}(\pi_{l-1})}^2 \vertiii{\hat{g}(\pi_{l-1})}^2
\end{align*}
for $\theta > 1$. But this is bounded by
\[ \left( \sum_{l \geq 1} (l-1)^{3/2 + 2\beta_w - \beta_0 + \alpha} \vertiii{\hat{f}(\pi_{l-1})}^2  \right) \left( \sum_{l \geq 1} (l-1)^{3/2 + 2\beta_w - \beta_0 + \beta} \vertiii{\hat{g}(\pi_{l-1})}^2 \right), \]
where $\alpha + \beta = \theta > 1$, $\alpha \geq 0$, $\beta \geq 0$, which gives precisely the same conclusion as in the $n \neq 0$ case.
\end{proof}

\section{Proof of \Cref{cor:inhomogeneous_basic_estimate}}
\label{sec:proof_of_first_corollary}

In this section we show that \Cref{cor:inhomogeneous_basic_estimate} follows from estimate \eqref{basic_estimate_with_extra_epsilons}, i.e.
\begin{equation}
	\label{null_form_estimates}
	\int_{-\pi}^{\pi} \|Q_0(\phi,\psi)\|^2_{H^{\epsilon}(\mathbb{S}^3)} \, \d t \la \| f _1 \|^2_{H^{1+\epsilon}(\mathbb{S}^3)} \| g_1 \|^2_{H^{\epsilon}(\mathbb{S}^3)}
\end{equation}
for solutions $\phi$, $\psi$ of
\begin{align}
\label{free_equations_with_time_data}
\begin{split}
	& \Box_{\mathbb{R} \times \mathbb{S}^3} \phi + \phi = 0, \qquad \qquad \qquad \! \Box_{\mathbb{R} \times \mathbb{S}^3} \psi + \psi = 0, \qquad \text{with} \\
	& (\phi, \partial_t \phi)|_{t=0} = (0, f_1), \qquad \qquad (\psi, \partial_t \psi)|_{t=0} = (0, g_1).
\end{split}	
\end{align}
Note that for \eqref{null_form_estimates} it is obvious that $[-\pi,\pi]$ can be replaced with $[0,T]$ for any $T>0$.

\subsection{First reduction} 

We first show that \eqref{null_form_estimates} implies the estimate
\begin{equation}
\label{derivative_estimates}
\int_{-\pi}^\pi \| Q_0(\phi,\psi) \|^2_{H^{1+\epsilon}(\mathbb{S}^3)} \, \d t + \int_{-\pi}^{\pi} \| \partial_t Q_0(\phi,\psi) \|^2_{H^\epsilon(\mathbb{S}^3)} \, \d t \la \| f_1 \|^2_{H^{1+\epsilon}(\mathbb{S}^3)} \| g_1 \|^2_{H^{1+\epsilon}(\mathbb{S}^3)}
\end{equation}
for solutions $\phi$, $\psi$ of \eqref{free_equations_with_time_data}. For the time derivative, we note\footnote{Note that $(1-\slashed{\Delta})^{1/2}$ commutes with $(\Box_{\mathbb{R}\times\mathbb{S}^3)} + 1) = \partial_t^2 + (1-\slashed{\Delta})$, and $\| \slashgrad f \|_{L^2} \simeq \| (1-\slashed{\Delta})^{1/2} f \|_{L^2} $.} that $\partial_t \phi^{\pm} = \pm i (1-\slashed{\Delta})^{1/2} \phi^\pm$, where $\phi^\pm$ are the positive and negative frequency parts of $\phi$ (see \eqref{discreteposnegfrequencies}), and $\phi$ is given by $\phi = \frac{1}{2i}(\phi^+ - \phi^-)$. Then, since \eqref{null_form_estimates} holds for all signs $Q_0(\phi^\pm, \psi^\pm)$, the time derivative estimate follows from
\begin{align*}
	\int_{-\pi}^{\pi} \| \partial_t Q_0(\phi, \psi) \|^2_{H^\epsilon(\mathbb{S}^3)}  \, \d t & \la \sum_{\pm} \| Q_0(\partial_t \phi^\pm, \psi^\pm) \|^2_{L^2([-\pi,\pi]; H^\epsilon(\mathbb{S}^3))} \\
	& \phantom{\la} + \sum_\pm \| Q_0(\phi^\pm, \partial_t \psi^\pm) \|^2_{L^2([-\pi,\pi];H^\epsilon(\mathbb{S}^3))} \\
	& \la \sum_\pm \| Q_0((1-\slashed{\Delta})^{1/2} \phi^\pm, \psi^\pm) \|^2_{L^2([-\pi,\pi];H^\epsilon(\mathbb{S}^3))} \\ 
	& \phantom{\la} + \sum_\pm \| Q_0(\phi^\pm, (1-\slashed{\Delta})^{1/2} \psi^\pm) \|^2_{L^2([-\pi,\pi]; H^\epsilon(\mathbb{S}^3))} \\
	& \la \| Q_0((1-\slashed{\Delta})^{1/2}\phi, \psi) \|^2_{L^2([-\pi,\pi];H^\epsilon(\mathbb{S}^3))} \\
	& \phantom{\la} + \| Q_0(\phi, (1-\slashed{\Delta})^{1/2} \psi) \|^2_{L^2([-\pi,\pi];H^\epsilon(\mathbb{S}^3))} \\
	& \la \| f_1 \|^2_{H^{1+\epsilon}(\mathbb{S}^3)} \| g_1 \|^2_{H^{1+\epsilon}(\mathbb{S}^3)}.
\end{align*}
In particular, this also shows that if $\phi$ has data $(f_0, f_1)$ and $\psi$ has data $(0, g_1)$, then writing $\phi = \partial_t \phi_0 + \phi_1$ for $\phi_i$ such that $\Box_{\mathbb{R} \times \mathbb{S}^3} \phi_i + \phi_i = 0$, $(\phi_i, \partial_t \phi_i)|_{t=0} = (0, f_i)$, and using the positive/negative frequency splitting we have similarly
\begin{align}
	\label{null_form_estimates_one_side_full_data}
	\begin{split}
		\| Q_0(\phi,\psi) \|_{L^2([-\pi,\pi];H^\epsilon(\mathbb{S}^3))} & \la \| Q_0(\partial_t \phi_0, \psi )\|_{L^2([-\pi,\pi];H^\epsilon(\mathbb{S}^3))} +  \| Q_0(\phi_1, \psi )\|_{L^2([-\pi,\pi];H^\epsilon(\mathbb{S}^3))} \\
		& \la \left(\| f_0 \|_{H^{1+\epsilon}(\mathbb{S}^3)} + \|f_1\|_{H^\epsilon(\mathbb{S}^3)} \right) \| g_1 \|_{H^{1+\epsilon}(\mathbb{S}^3)}.
	\end{split}
\end{align}

For the spatial derivatives $\slashgrad_i$, we differentiate the equations \eqref{free_equations_with_time_data}. Unlike in the case of $\mathbb{R}^{1+3}$, the situation here is slightly complicated by the fact that $\mathbb{S}^3$ is not flat, so the spatial derivatives $\slashed{\nabla}_i$ do not commute with the Laplacian $\slashed{\Delta}$. This means that $\slashgrad_i \phi$ now satisfies the perturbed equation $\Box_{\mathbb{R} \times \mathbb{S}^3} (\slashgrad_i \phi) + 3 (\slashgrad_i \phi) = 0$ with data $(0, \slashgrad_i f_1)$, and similarly for $\slashgrad_i \psi$. To deal with this, we use Duhamel's principle to write $\slashgrad_i \phi$ as the sum
\[ \slashgrad_i \phi = \Phi_1 + \Phi_2, \]
where $\Phi_1$ solves
\[ \Box_{\mathbb{R} \times \mathbb{S}^3} \Phi_1 + \Phi_1 = 0 \qquad \text{with} \qquad (\Phi_1, \partial_t \Phi_1)|_{t=0} = (0, \slashgrad_i f_1), \]
and
\[ \Phi_2(t,x) = \int_0^t \Phi_2^s(t-s,x) \, \d s, \]
where $\Phi_2^s$ solves
\[ \Box_{\mathbb{R} \times \mathbb{S}^3} \Phi_2^s + \Phi_2^s = 0 \qquad \text{with} \qquad (\Phi_2^s, \partial_t \Phi_2^s)|_{t=0} = (0, -2\slashgrad_i \phi(s,x)), \]
so that $\Phi_2$ solves
\[ \Box_{\mathbb{R} \times \mathbb{S}^3} \Phi_2 + \Phi_2 = -2 \slashgrad_i \phi \qquad \text{with} \qquad (\Phi_2, \partial_t \Phi_2)|_{t=0} = (0,0). \]
Then the bounds involving $\Phi_1$ (and the analogous $\Psi_1$) are immediate,
\begin{align*}
	\int_{-\pi}^{\pi} \| \slashgrad Q_0(\phi, \psi) \|^2_{H^\epsilon(\mathbb{S}^3)} \, \d t & \la \| Q_0(\Phi_1, \psi) \|^2_{L^2([-\pi,\pi];H^\epsilon(\mathbb{S}^3))} + \| Q_0(\Phi_2, \psi)\|^2_{L^2([-\pi,\pi];H^\epsilon(\mathbb{S}^3))} \\
	& \phantom{\la} + \| Q_0(\phi, \Psi_1) \|^2_{L^2([-\pi,\pi];H^\epsilon(\mathbb{S}^3))} + \| Q_0(\phi, \Psi_2) \|^2_{L^2([-\pi,\pi ];H^\epsilon(\mathbb{S}^3))} \\
	& \la \| Q_0(\Phi_2, \psi)\|^2_{L^2([-\pi,\pi];H^\epsilon(\mathbb{S}^3))} + \| Q_0(\phi, \Psi_2) \|^2_{L^2([-\pi,\pi ];H^\epsilon(\mathbb{S}^3))} \\
	& \phantom{\la} +  \| f_1 \|^2_{H^{1+\epsilon}(\mathbb{S}^3)} \| g_1 \|^2_{H^{1+\epsilon}(\mathbb{S}^3)},
\end{align*}
so it remains to bound the forms involving $\Phi_2$ and $\Psi_2$. This is even easier; we have, using that $\Phi_2^s(0,x) = 0$ for all $s$,
\[ Q_0(\Phi_2, \psi) = \int_0^t Q_0(\Phi_2^s(t-s,x), \psi(t,x)) \, \d s, \]
so, using \eqref{null_form_estimates_one_side_full_data} followed by the energy inequality for $\Phi_2^s$ and then the energy inequality for $\phi$, one gets
\begin{align*}
	\| Q_0(\Phi_2, \psi) \|_{L^2([-\pi,\pi];H^\epsilon(\mathbb{S}^3))} & \la \int_{-\pi}^{\pi} \| Q_0(\Phi_2^s(t-s,x), \psi(t,x) ) \|_{L^2([-\pi,\pi];H^\epsilon(\mathbb{S}^3))} \, \d s	 \\
	& \la \int_{-\pi}^{\pi} \left( \|\Phi_2^s(-s,\cdot)\|_{H^{1+\epsilon}(\mathbb{S}^3)} + \|\partial_t\Phi_2^s(-s,\cdot)\|_{H^\epsilon(\mathbb{S}^3)} \right) \| g_1 \|_{H^{1+\epsilon}(\mathbb{S}^3)} \, \d s \\
	& \la \int_{-\pi}^{\pi} \| \slashgrad \phi(s,\cdot) \|_{H^\epsilon(\mathbb{S}^3)} \| g_1 \|_{H^{1+\epsilon}(\mathbb{S}^3)} \, \d s \\
	& \la \int_{-\pi}^{\pi} \| f_1\|_{H^\epsilon(\mathbb{S}^3)} \| g_1 \|_{H^{1+\epsilon}(\mathbb{S}^3)} \, \d s \\
	& \la \| f_1 \|_{H^\epsilon(\mathbb{S}^3)} \| g_1 \|_{H^{1+\epsilon}(\mathbb{S}^3)}.
\end{align*}
The estimate for $Q_0(\phi, \Psi_2)$ is analogous. This establishes the implication \eqref{null_form_estimates}$\implies$\eqref{derivative_estimates}.

\subsection{Second reduction} \label{sec:second_reduction}

We now show that \eqref{derivative_estimates} implies \Cref{cor:inhomogeneous_basic_estimate}. To do this, we must incorporate solutions $\phi$ with general data and right-hand side $(f_0,f_1,F)$. Such a solution may be written as the sum of solutions with $(f_0,0,0)$, $(0,f_1,0)$, and $(0,0,F)$, so consider first the case $F = G = 0$. If $\phi_i$, $i=1,2$, solve
\[ \Box_{\mathbb{R} \times \mathbb{S}^3} \phi_i + \phi_i = 0 \qquad \text{with} \qquad (\phi_i, \partial_t \phi_i)|_{t=0} = (0, f_i), \]
and similarly for $\psi_i$, then, as before, $\phi = \partial_t \phi_0 + \phi_1$ solves $\Box_{\mathbb{R} \times \mathbb{S}^3} \phi + \phi = 0$ with data $(f_0, f_1)$, and similarly for $\psi = \partial_t \psi_0 + \psi_1$. Then \eqref{derivative_estimates} implies \Cref{cor:inhomogeneous_basic_estimate} with $F = G = 0$ by the same argument used to deduce \eqref{null_form_estimates_one_side_full_data}. To deduce the full estimate, it therefore suffices to treat the case $f_0 = f_1 = g_0 = g_1 =0$, $F, G \neq 0$. By Duhamel's principle again, the solution with zero initial data and right-hand side $F$ can be written in terms of a solution with data $(0,f_1)$ and zero right-hand side, i.e. if $\Phi_s(t,x)$ satisfies 
\[
	\Box_{\mathbb{R} \times \mathbb{S}^3} \Phi_s + \Phi_s = 0 \qquad \text{with} \qquad (\Phi_s, \partial_t \Phi_s)|_{t=0} = (0, F(s,x)),	
\]
then
\[ \phi(t,x) = \int_0^t \Phi_s(t-s,x) \, \d s \]
solves $\Box_{\mathbb{R} \times \mathbb{S}^3} \phi + \phi = F$ with data $(f_0,f_1) = (0,0)$. Similarly, we have $\Psi_s$, where
\[ \psi(t,x) = \int_0^t \Psi_s(t-s,x) \, \d s \]
satisfies $\Box_{\mathbb{R} \times \mathbb{S}^3} \psi + \psi = G$ with data $(g_0,g_1) = (0,0)$. Then it is easy to check, similarly to before, that
\[ Q_0(\phi,\psi)(t,x) = \int_0^t \int_0^t Q_0(\Phi_s(t-s,x), \Psi_{r}(t-r,x) ) \, \d s \, \d r. \]
For spatial derivatives it then follows trivially that
\[ \slashed{\nabla} Q_0(\phi,\psi)(t,x) = \int_0^t \int_0^t \slashed{\nabla} Q_0(\Phi_s(t-s,x), \Psi_{r}(t-r,x) ) \, \d s \, \d r. \]
We can then estimate, using the already-established \Cref{cor:inhomogeneous_basic_estimate} with $F= G = 0$,
\begin{align}
	\label{spatial_derivative_estimate_null_form}
	\begin{split}
	\| \slashed{\nabla} Q_0(\phi,\psi) \|_{L^2([-\pi,\pi];H^{\epsilon}(\mathbb{S}^3))} & \la \int_{-\pi}^\pi \int_{-\pi}^{\pi} \| \slashed{\nabla} Q_0(\Phi_s(t-s,x), \Psi_r(t-r,x) \|_{L^2([-\pi,\pi];H^{\epsilon}(\mathbb{S}^3))} \, \d s \, \d r \\
	& \la \int_{-\pi}^{\pi} \bigg( \| \Phi_s(-s,\cdot) \|_{H^{2+\epsilon}(\mathbb{S}^3)} + \| \partial_t \Phi_s(-s, \cdot) \|_{H^{1+\epsilon}(\mathbb{S}^3)} \bigg) \d s \\
	& \times \int_{-\pi}^{\pi} \bigg( \| \Psi_r(-r,\cdot) \|_{H^{2+\epsilon}(\mathbb{S}^3)} + \| \partial_t \Psi_r(-r,\cdot) \|_{H^{1+\epsilon}(\mathbb{S}^3)} \bigg) \d r \\
	& \la \int_{-\pi}^{\pi} \| F(s,\cdot) \|_{H^{1+\epsilon}(\mathbb{S}^3)} \, \d s \int_{-\pi}^{\pi} \| G(r,\cdot) \|_{H^{1+\epsilon}(\mathbb{S}^3)} \, \d r,
	\end{split}
\end{align}
where in the last line we used the energy inequality for $\Phi_s$ and $\Psi_r$, in evolution from time $-s$ (resp. $-r$) to time zero. Combining the three cases $(f_0,0,0)$, $(0,f_1,0)$ and $(0,0,F)$ (and the same for $g_i$, $G$), this establishes the implication \eqref{derivative_estimates}$\implies$\Cref{cor:inhomogeneous_basic_estimate}.

\subsection{Estimate for the time derivative}

For the time derivative, on the other hand, one gets
\begin{align*}
	\partial_t Q_0(\phi, \psi)(t,x) &= F(t,x) (\partial_t \psi)(t,x) + G(t,x) (\partial_t \phi)(t,x) \\
	& +  \int_0^t \int_0^t \partial_t Q_0(\Phi_s(t-s,x), \Psi_r(t-r,x)) \, \d s \, \d r .
\end{align*}
In this, the integrated time derivative of $Q_0$ can be controlled in the same way (\eqref{spatial_derivative_estimate_null_form}) as the spatial derivatives of $Q_0$, but the appearance of the quadratic terms $F \partial_t \psi$ and $G \partial_t \phi$ and the fact that $\epsilon$ must be positive requires the Kato--Ponce inequality \cite{KatoPonce1988,KenigPonceVega1993}\footnote{See also \cite{BaeBiswas2015,MuscaluSchlag2013,ChristWeinstein1991,GulisashviliKon1996,Grafakos2012,GrafakosOh2013,Li2018}. Note that we need \eqref{Kato_Ponce_inequality} on $\mathbb{S}^3$, which can be obtained from the Euclidean setting using a partition of unity.}
\begin{equation}
	\label{Kato_Ponce_inequality}
	\| J^\epsilon(fg) \|_{L^r} \la \| f \|_{L^{p_1}} \| J^\epsilon g \|_{L^{q_1}} + \| J^\epsilon f \|_{L^{p_2}} \| g \|_{L^{q_2}},	
\end{equation}
where $J^\epsilon = (1-\slashed{\Delta})^{\epsilon/2}$, $\frac{1}{r} = \frac{1}{p_1} + \frac{1}{q_1} = \frac{1}{p_2} + \frac{1}{q_2}$, and $1<r<\infty$, $1< p_1$, $p_2$, $q_1$, $q_2 < \infty$. We then have
\begin{align*}
	\| F \partial_t \psi \|_{H^{\epsilon}(\mathbb{S}^3)} & = \| J^\epsilon (F \partial_t \psi) \|_{L^2(\mathbb{S}^3)} \\
	& \la \|F\|_{L^3(\mathbb{S}^3)} \| J^\epsilon(\partial_t \psi) \|_{L^6(\mathbb{S}^3)} + \| J^\epsilon(F) \|_{L^3(\mathbb{S}^3)} \| \partial_t \psi \|_{L^6(\mathbb{S}^3)} \\
	& \la \| F \|_{H^{1+\epsilon}(\mathbb{S}^3)} \| \partial_t \psi \|_{H^{1+\epsilon}(\mathbb{S}^3)},
\end{align*}
so, using the Sobolev embedding $H^1(\mathbb{S}^3) \hookrightarrow L^6(\mathbb{S}^3)$ and the energy inequality for $\psi$,
\begin{align*}
	\| F \partial_t \psi \|_{L^2([-\pi,\pi];H^{\epsilon}(\mathbb{S}^3))} & \la \|F\|_{L^2([-\pi,\pi];H^{1+\epsilon}(\mathbb{S}^3))} \| \partial_t \psi \|_{L^\infty([-\pi,\pi];H^{1+\epsilon}(\mathbb{S}^3))} \\
	& \la \|F\|_{L^2([-\pi,\pi];H^{1+\epsilon}(\mathbb{S}^3))} \\
	& \times \left( \| g_0 \|_{H^{2+\epsilon}(\mathbb{S}^3)} + \| g_1 \|_{H^{1+\epsilon}(\mathbb{S}^3)} + \int_{-\pi}^{\pi} \| G(t,\cdot)\|_{H^{1+\epsilon}(\mathbb{S}^3)} \, \d t \right) \\
	& \la \left( \| f_0 \|_{H^{2+\epsilon}(\mathbb{S}^3)} + \| f_1 \|_{H^{1+\epsilon}(\mathbb{S}^3)} + \| F \|_{L^2([-\pi,\pi];H^{1+\epsilon}(\mathbb{S}^3))} \right) \\
	& \hspace{2pt} \times \left( \| g_0 \|_{H^{2+\epsilon}(\mathbb{S}^3)} + \| g_1 \|_{H^{1+\epsilon}(\mathbb{S}^3)} + \| G \|_{L^2([-\pi,\pi];H^{1+\epsilon}(\mathbb{S}^3))} \right),
\end{align*}
as well as the same statement for $G \partial_t \phi$. This then implies the estimate
\begin{align}
\begin{split}
\label{main_estimate_with_time_derivative}
 \| Q_0(\phi,\psi) \|_{L^2([-\pi,\pi];H^{1+\epsilon}(\mathbb{S}^3))} &+ \| \partial_t Q_0(\phi,\psi) \|_{L^2([-\pi,\pi];H^{\epsilon}(\mathbb{S}^3))}  \\
	& \la \big( \| f_0 \|_{H^{2+\epsilon}(\mathbb{S}^3)} + \| f_1 \|_{H^{1+\epsilon}(\mathbb{S}^3)} + \| F \|_{L^2([-\pi,\pi];H^{1+\epsilon}(\mathbb{S}^3))} \big) \\
	& \hspace{1pt} \times \big( \| g_0 \|_{H^{2+\epsilon}(\mathbb{S}^3)} + \| g_1 \|_{H^{1+\epsilon}(\mathbb{S}^3)} + \| G \|_{L^2([-\pi,\pi];H^{1+\epsilon}(\mathbb{S}^3))} \big) .
\end{split}
\end{align}

\section{Proof of the Weighted Estimate on $\mathbb{R}^{1+3}$} \label{sec:conformal_estimate}

Finally, we turn to the proof of the conformally weighted estimate of \Cref{cor:weighted_estimates}. Here we use the conformal invariance of the wave equation; if $\phi$ satisfies 
\[ \Box_{\mathbb{R} \times \mathbb{S}^3} \phi + \phi = 0 \]
on $\mathbb{R} \times \mathbb{S}^3$ with data $(f_0, f_1)$, then $\tilde{\phi} = \Omega \phi$ satisfies
\[ \Box_{\mathbb{R}^{1+3}} \tilde{\phi} = 0 \]
on $\mathbb{R}^{1+3}$ with data $(\tilde{f}_0, \tilde{f}_1)=(\omega f_0, \omega^2 f_1)$, where $\omega = 2(1+ \tilde{r}^2)^{-1} = \Omega|_{\tilde{t}=0}$, and $\Omega$ is the conformal factor
\begin{align}
\label{conformal_factor}
\begin{split}
 \Omega &= 2  \left( 1 + (\tilde{t} - \tilde{r})^2 \right)^{-1/2} \left( 1 + (\tilde{t} + \tilde{r})^2 \right)^{-1/2} \\
& = 2 \cos \left( \frac{t-z}{2} \right) \cos \left( \frac{t+z}{2} \right),
\end{split}
\end{align}
where $\tilde{t}$, $\tilde{r}$ are the usual coordinates on $\mathbb{R}^{1+3}$ and $z = \arctan(\tilde{t} + \tilde{r}) - \arctan(\tilde{t} - \tilde{r})$ is one of the angles on $\mathbb{S}^3$, the metric $g_{\mathbb{S}^3}$ being given by $\d z^2 + (\sin^2 z) g_{\mathbb{S}^2}$. Note that $\Omega$ is a smooth function on $\mathbb{R} \times \mathbb{S}^3$, and we have
\[ Q_0(\Omega, \Omega) = 2 \Omega \sin \left( \frac{t-z}{2} \right) \sin \left( \frac{t+z}{2} \right). \]
Writing $\widetilde{Q}_0(\tilde{\phi},\tilde{\psi})$ to denote the null form on $(\mathbb{R}^{1+3}, \eta_{ab})$,
\[ \widetilde{Q}_0(\tilde{\phi},\tilde{\psi}) = \tilde{\nabla}_a \tilde{\phi} \tilde{\nabla}^a \tilde{\psi}, \]
a calculation shows 
\begin{equation}
\label{conformal_null_form}
	\widetilde{Q}_0(\tilde{\phi}, \tilde{\psi}) = \Omega^4 Q_0(\phi, \psi) + \Omega^3 \phi Q_0(\Omega, \psi) + \Omega^3 \psi Q_0(\Omega, \phi) + \Omega^2 \phi \psi Q_0(\Omega, \Omega),
\end{equation}
or, taking into account the above expression for $Q_0(\Omega, \Omega)$,
\begin{equation}
\label{conformal_null_form_2}
 \widetilde{Q}_0(\tilde{\phi}, \tilde{\psi}) = \Omega^4 Q_0(\phi, \psi) + \Omega^3 \left( \mathfrak{a}^\mu \phi \nabla_\mu \psi + \mathfrak{b}^\mu \psi \nabla_\mu \phi + \mathfrak{c} \phi \psi \right)
\end{equation}
for a smooth function $\mathfrak{c}$ and smooth vector fields $\mathfrak{a}^\mu$, $\mathfrak{b}^\mu$ on $\mathbb{R} \times \mathbb{S}^3$. Note that the standard foliation of $\mathbb{R} \times \mathbb{S}^3$ which was used for our estimates in the previous sections corresponds to a hyperboloidal foliation on $\mathbb{R}^{1+3}$, given by the level sets $\Sigma_t \simeq \mathbb{R}^3$ of the function
\[ t = \arctan(\tilde{t} + \tilde{r}) + \arctan(\tilde{t} - \tilde{r}) \in (-\pi, \pi), \]
with unit normal (with respect to $\eta_{ab} = \d \tilde{t}^2 - \d \tilde{r}^2 - \tilde{r}^2 g_{\mathbb{S}^2}$)
\[ T^a \partial_a = \left( 1+ 2(\tilde{t}^2 + \tilde{r}^2) + (\tilde{t}^2 - \tilde{r}^2)^2 \right)^{-1} \left( 2 (1+\tilde{t}^2 + \tilde{r}^2) \frac{\partial}{\partial \tilde{t}} + 4 \tilde{t} \tilde{r} \frac{\partial}{\partial \tilde{r}} \right). \]
The metric on $\mathbb{R} \times \mathbb{S}^3$ is related to the Minkowski metric $\eta_{ab}$ by
\[ g^{\mathbb{R} \times \mathbb{S}^3}_{ab} = \d t^2 - g_{\mathbb{S}^3} = \Omega^2 \eta_{ab}, \]
so writing $\Omega^{-2} g_{\mathbb{S}^3} \overset{\text{def}}{=} g_{\Sigma_t}$, we have $\dvol_{\Sigma_t} = \Omega^{-3} \dvol_{\mathbb{S}^3}$, and
\[ \eta_{ab} = \Omega^{-2} \d t^2 - g_{\Sigma_t}, \]
where the volume form of $g^{\mathbb{R}\times \mathbb{S}^3}$ is related to the volume form of $\eta$ by $\dvol_\eta = \Omega^{-4} \d t \wedge \dvol_{\mathbb{S}^3} = \Omega^{-1} \d t \wedge \dvol_{\Sigma_t}$. From \eqref{conformal_null_form_2}, we have
\begin{align*}
	\Omega^{-6} | \widetilde{Q}_0(\tilde{\phi},\tilde{\psi})|^2 &\la \Omega^2 |Q_0(\phi,\psi)|^2 + |\phi|^2(|\partial_t \psi|^2 + |\slashgrad \psi|^2) + |\psi|^2(|\partial_t \phi|^2 + |\slashgrad \phi|^2) + |\phi|^2 |\psi|^2, 	
\end{align*}
so integrating with respect to $\dvol_{\mathbb{S}^3}$, at fixed $t$,
\begin{align*}
	\int_{\Sigma_t} \Omega^{-3} |\widetilde{Q}_0(\tilde{\phi},\tilde{\psi})|^2 \dvol_{\Sigma_t} & \la \| Q_0(\phi,\psi) \|^2_{L^2(\mathbb{S}^3)} (t) \\
	& + \int_{\{ t \} \times \mathbb{S}^3} \left( |\phi|^2(|\partial_t \psi|^2 + |\slashgrad \psi|^2) + |\psi|^2(|\partial_t \phi|^2 + |\slashgrad \phi|^2) + |\phi|^2 |\psi|^2 \right) \dvol_{\mathbb{S}^3}	\\
	& \la \| Q_0(\phi,\psi) \|^2_{L^2(\mathbb{S}^3)} (t) + \| \phi \|^2_{L^\infty(\mathbb{S}^3)}(t) \left( \| \partial_t \psi \|^2_{L^2(\mathbb{S}^3)}(t) + \| \psi \|^2_{H^1(\mathbb{S}^3)}(t) \right) \\
	& + \| \psi \|^2_{L^6(\mathbb{S}^3)}(t) \left( \| \partial_t \phi \|^2_{L^3(\mathbb{S}^3)}(t) + \| \slashgrad \phi \|^2_{L^3(\mathbb{S}^3)}(t) \right),
\end{align*}
where in the last line we used H\"older's inequality with $\frac{1}{p} + \frac{1}{q} = \frac{1}{6} + \frac{1}{3} = \frac{1}{2}$ to control the terms of the form $|\psi|^2 |\partial \phi |^2$. By the compactness of $\mathbb{S}^3$ and using the embeddings $H^1(\mathbb{S}^3) \hookrightarrow L^6(\mathbb{S}^3)$ and $H^2(\mathbb{S}^3) \hookrightarrow L^\infty(\mathbb{S}^3)$ we obtain, after using the energy inequality \eqref{energy_inequality} for $\phi$ and $\psi$,
\begin{align*}
	\int_{\Sigma_t}	\Omega^{-3} |\widetilde{Q}_0(\tilde{\phi}, \tilde{\psi})|^2 \dvol_{\Sigma_t} & \la \| Q_0 (\phi,\psi)\|^2_{L^2(\mathbb{S}^3)}(t) \\
	& + \left( \| f_0 \|^2_{H^2(\mathbb{S}^3)} + \| f_1 \|^2_{H^1(\mathbb{S}^3)} \right) \left( \| g_0 \|^2_{H^1(\mathbb{S}^3)} + \| g_1 \|^2_{L^2(\mathbb{S}^3)} \right).
\end{align*}
Hence integrating in $-\pi \leq t \leq \pi$ and using \eqref{basic_estimate_2} gives
\begin{align*}
	\| \Omega^{-1} \widetilde{Q}_0(\tilde{\phi}, \tilde{\psi}) \|^2_{L^2(\mathbb{R}^4)} & \la \left( \| f_0 \|^2_{H^2(\mathbb{S}^3)} + \| f_1 \|^2_{H^1(\mathbb{S}^3)} \right) \left( \| g_0 \|^2_{H^{1+\epsilon}(\mathbb{S}^3)} + \| g_1 \|^2_{H^{\epsilon}(\mathbb{S}^3)} \right),
\end{align*}
or, using the definition of weighted Sobolev spaces \eqref{weighted_spaces_definition} and noting that $(\tilde{f}_0, \tilde{f}_1) = (\omega f_0, \omega^2 f_1)$ and similarly for $(\tilde{g}_0, \tilde{g}_1)$,
\[ \left\| \langle \tilde{t} - \tilde{r} \rangle \langle \tilde{t} + \tilde{r} \rangle \widetilde{Q}_0(\tilde{\phi},\tilde{\psi}) \right\|^2_{L^2(\mathbb{R}^4)} \la \left( \| \tilde{f}_0 \|^2_{W^2_1(\mathbb{R}^3)} + \| \tilde{f}_1 \|^2_{W^1_2(\mathbb{R}^3)} \right) \left( \| \tilde{g}_0 \|^2_{W^{1+\epsilon}_{\epsilon}(\mathbb{R}^3)} + \| \tilde{g}_1 \|^2_{W^{\epsilon}_{1+\epsilon}(\mathbb{R}^3)} \right). \] 
This completes the proof of \Cref{cor:weighted_estimates}.

\begin{remark}
	It is worth noting that in the above calculation we did not use the full strength of the estimate on $\| Q_0(\phi,\psi)\|_{L^2}$, only the weaker estimate on $\|\Omega Q_0(\phi,\psi)\|_{L^2} \la \|Q_0(\phi,\psi)\|_{L^2}$ in order to avoid weights of inverse powers of $\Omega$ in front of the lower order terms. But handling these would only need weighted energy estimates for the linear wave equation, so we expect that the above calculation can be improved to obtain an estimate for $\| \langle \tilde{t} - \tilde{r} \rangle^2 \langle \tilde{t} + \tilde{r} \rangle^2 \widetilde{Q}_0(\tilde{\phi}, \tilde{\psi}) \|_{L^2(\mathbb{R}^4)}$.
\end{remark}

\subsection*{Acknowledgements} The author thanks Sergiu Klainerman for useful discussions and an anonymous referee for a careful reading of the manuscript, particularly for suggesting an improvement to the proof of \Cref{lem:aux_2}.

\appendix

\section{Harmonic Analysis on Lie Groups} \label{sec:harmonicanalygroups}

\subsection{Peter--Weyl Theorem}

Here we recap briefly the fundamentals of harmonic analysis on Lie groups. The proofs of the statements below can be found in standard texts, e.g. \cite{Faraut2008}. Let $\mathrm{G}$ be a compact matrix Lie group. A \emph{representation} of $\mathrm{G}$ on a finite dimensional complex vector space $V$ is a group homomorphism $\pi : \mathrm{G} \to \mathrm{GL}(V)$ such that the map $g \mapsto \pi(g) v$ is continuous for all $v \in V$. We say two representations $(\pi_1, V_1)$ and $(\pi_2, V_2)$ are \emph{equivalent} if there exists a continuous isomorphism $A :V_1 \to V_2$ such that $A \pi_1(g) = \pi_2(g) A$ for all $g \in \mathrm{G}$. In this case we write $\pi_1 \simeq \pi_2$. Let $(\pi, V)$ be a representation of $\mathrm{G}$ and let $W \subset V$ be a subspace of $V$. The space $W$ is said to be an \emph{invariant subspace} of the representation $\pi$ if $\pi(\mathrm{G}) W \subset W$. One then defines an \emph{irreducible} representation to be a representation which has no non-trivial invariant subspaces.

Next suppose $V$ is a complex vector space equipped with a Hermitian inner product $\langle \cdot, \cdot \rangle$. Recall that an operator $A$ on $V$ is \emph{unitary} if $\langle A u , Aw \rangle = \langle v, w \rangle$ for all $v,w \in V$. We say that a representation $\pi$ of $\mathrm{G}$ on $V$ is \emph{unitary} if for every $g \in \mathrm{G}$ the operator $\pi(g)$ is unitary.

Every compact Lie group $\mathrm{G}$ admits a unique (up to multiplication by a positive constant) left-invariant measure $\mu$ called the \emph{Haar measure}. That is, the measure $\mu$ satisfies
\[ \int_{\mathrm{G}} f(kx) \, \d \mu(x) = \int_{\mathrm{G}} f(x) \, \d \mu(x) \]
for every measurable function $f : \mathrm{G} \to \mathbb{C}$ and all $k \in \mathrm{G}$. We fix the Haar measure by the normalization
\[ \int_{\mathrm{G}} \d \mu(x) = 1. \]
The Haar measure on any compact connected Lie group $\mathrm{G}$ is also right-invariant.

\begin{definition} Let $\hat{\mathrm{G}}$ denote the set of equivalence classes of unitary irreducible representations of a compact Lie group $\mathrm{G}$. The set $\hat{\mathrm{G}}$ is called the \emph{unitary dual} of $\mathrm{G}$.	
\end{definition}

\begin{definition} Let $\pi : \mathrm{G} \to \mathrm{GL}(V)$ be a unitary $\mathbb{C}$-linear irreducible representation of $\mathrm{G}$ on a complex vector space $V$ of dimension $d_\pi$. Let $\{ e_i \}_{1 \leq i \leq d_\pi}$ be an orthonormal basis for $V$. We call
\[ \pi_{ij}(g) \overset{\text{def}}{=} \langle \pi(g) e_j, e_i \rangle \]
the \emph{matrix coefficients} of $\pi$ with respect to $\{ e_i \}_i$.	
\end{definition}

\begin{proposition}[Schur orthogonality] For $u,v,u',v' \in V$ one has the orthogonality relations
\[ \int_{\mathrm{G}} \langle \pi(g) u, v \rangle \overline{\langle \pi(g) u', v' \rangle} \, \d \mu(g) = \frac{1}{d_\pi} \langle u, u' \rangle \overline{\langle v, v' \rangle}. \]
\end{proposition}

As a consequence of Schur orthogonality, it is easy to check that the matrix coefficients of a unitary $\mathbb{C}$-linear irreducible representation $\pi$ are orthogonal in $L^2(\mathrm{G})$,
\begin{equation} \label{matrixcoefficientorthogonality} \langle \pi_{ij}, \pi_{kl} \rangle_{L^2(\mathrm{G})} \overset{\text{def}}{=} \int_{\mathrm{G}} \pi_{ij}(g) \overline{\pi_{kl}(g)} \, \d \mu(g) = \frac{1}{d_\pi} \delta_{ik} \delta_{jl}. \end{equation}
Let $\mathcal{M}_\pi$ denote the subspace of $L^2(\mathrm{G})$ spanned by the matrix coefficients $\pi_{ij}$ for $1 \leq i,j \leq d_\pi$. It is not hard to see that $\mathcal{M}_\pi$ is independent of the choice of orthonormal basis for $V$, and that, on account of \eqref{matrixcoefficientorthogonality},
\[ \left\{ \sqrt{d_\pi} \pi_{ij} \, : \, 1 \leq i, j \leq d_\pi \right\} \]
is an orthonormal basis for $\mathcal{M}_\pi \subset L^2(\mathrm{G})$.

\begin{proposition} If $\pi \not\simeq \pi'$ are two inequivalent unitary irreducible representations of $\mathrm{G}$, then $\mathcal{M}_\pi$ and $\mathcal{M}_{\pi'}$ are orthogonal in $L^2(\mathrm{G})$,
\[ \int_{\mathrm{G}} \langle \pi(g) u, v \rangle \overline{\langle \pi'(g) u', v' \rangle'} \, \d \mu(g) = 0 \]
for all $u,v \in V$, $u', v' \in V'$.	
\end{proposition}

\begin{theorem}[Peter--Weyl]
\[ L^2(\mathrm{G}) = \widehat{\bigoplus}_{\pi \in \hat{\mathrm{G}}} \mathcal{M}_\pi, \]
where $\widehat{\bigoplus}_\pi \mathcal{M}_\pi$ denotes the closure of $\bigoplus_\pi \mathcal{M}_\pi$ in $L^2(\mathrm{G})$.	
\end{theorem}

\begin{definition} Let $f \in L^1(\mathrm{G})$. For each $\pi \in \hat{\mathrm{G}}$ the \emph{Fourier coefficient} $\hat{f}(\pi)$ is the operator on $V$ (the complex vector space associated with the representation $\pi$)  defined by 
\[ \hat{f}(\pi) \overset{\text{def}}{=} \int_{\mathrm{G}} f(g) \pi(g^{-1}) \, \d \mu(g). \]
For matrix Lie groups $\hat{f}(\pi)$ can be thought of as a matrix.
\end{definition}

By the Schur orthogonality relations and the Peter--Weyl theorem,
\[ \mathcal{B} = \left\{ \sqrt{d_\pi} \pi_{ij}, \, \pi \in \hat{\mathrm{G}}, \, 1 \leq i, j \leq d_\pi \right\} \]
is a Hilbert basis for $L^2(\mathrm{G})$. It is not difficult to check that for $f \in L^2(\mathrm{G})$ we have
\[ \langle f, \pi_{ij} \rangle_{L^2(\mathrm{G})} = \langle \hat{f}(\pi) e_j, e_i \rangle_V = \hat{f}(\pi)_{ij}. \]
This leads to the following.

\begin{theorem}[Plancherel Theorem] Let $f \in L^2(\mathrm{G})$. Then
\[ f(g) = \sum_{\pi \in \hat{\mathrm{G}}} d_\pi \operatorname{Tr}(\hat{f}(g) \pi(g) ) \]
in $L^2(\mathrm{G})$, and furthermore
\[ \| f \|^2_{L^2(\mathrm{G})} = \sum_{\pi \in \hat{\mathrm{G}}} d_\pi \vertiii{ \hat{f}(\pi) }^2, \]
where $\vertiii{\hat{f}(\pi)}$ is the Frobenius norm of the matrix $\hat{f}(\pi)$ defined by $\vertiii{A}^2 = \operatorname{Tr}(A A^\dagger)$.	
\end{theorem}

\subsection{Sobolev spaces on \texorpdfstring{$\mathrm{SU}(2)$}{SU2}}

We identify $\mathbb{S}^3 \simeq \mathrm{SU}(2)$. The normalized Haar measure on $\mathrm{SU}(2)$ is the volume form of the standard metric on $\mathbb{S}^3$,
\begin{equation} \label{Haarmeasure} \d \mu = \frac{1}{2 \pi^2} \dvol_{\mathbb{S}^3}. \end{equation}
For convenience we will define the Sobolev norms on $\mathbb{S}^3$ with respect to the scaling \eqref{Haarmeasure} of the $\mathbb{S}^3$ volume form. The finite-dimensional irreducible representations $\pi_m : \mathrm{SU}(2) \to \mathrm{GL}(V_m)$ of $\mathrm{SU}(2)$ are labelled by $m \in \mathbb{N}_0$, one for each $m$, and have dimensions $(m+1)$. They are exactly the eigenfunctions of the metric Laplacian $\slashed{\Delta}$ on $\mathbb{S}^3$ (or equivalently the Casimir operator on $\mathrm{SU}(2)$),
\[ \slashed{\Delta} \pi_m = - m(m+2) \pi_m. \]
Since $\pi_m$ is a matrix of size $(m+1) \times (m+1)$, the eigenvalue $-m(m+2)$ has multiplicity $(m+1)^2$. The Schur orthogonality relations imply that
\[ \int_{\mathrm{SU}(2)} (\pi_m)_{ij}(x) \overline{(\pi_n)_{i'j'}(x)} \, \d \mu(x) = \frac{1}{(m+1)} \delta_{i i'} \delta_{j j'} \delta_{mn}, \]
and the Peter--Weyl and Plancherel theorems state that any $f \in L^2(\mathbb{S}^3)$ is expressible as its Fourier series
\[ f(x) = \sum_{m \geq 0} (m+1) \operatorname{Tr}(\hat{f}(\pi_m) \pi_m(x)), \]
interpreted as equality in $L^2(\mathbb{S}^3)$. Furthermore, the $L^2$ norm of $f$ is given by
\[ \| f \|^2_{L^2(\mathbb{S}^3)} = \sum_{m \geq 0} (m+1) \, \vertiii{ \hat{f}(\pi_m) }^2.  \]
One can similarly define the Sobolev spaces $H^k(\mathbb{S}^3)$ using the $\mathrm{SU}(2)$ Fourier coefficients. By Schur orthogonality,
\begin{align*} \| \slashgrad f \|^2_{L^2(\mathbb{S}^3)} & = \langle f, - \slashed{\Delta} f \rangle_{L^2(\mathbb{S}^3)} \\
& = \sum_{m \geq 0} \sum_{n \geq 0} (n+1)(m+1) \hat{f}(\pi_n)_{ij} \hat{f}(\pi_m)_{rs} \langle (\pi_n)_{ji}, -\slashed{\Delta} (\pi_m)_{sr} \rangle_{L^2(\mathbb{S}^3)}	 \\
& = \sum_{m \geq 0} \sum_{n \geq 0} (n+1) m(m+1)(m+2) \hat{f}(\pi_n)_{ij} \hat{f}(\pi_m)_{rs} \delta_{mn} \frac{1}{(m+1)} \delta_{js} \delta_{ir} \\
& = \sum_{m \geq 0} m(m+1)(m+2) \, \vertiii{ \hat{f}(\pi_m) }^2,
\end{align*}
so
\begin{align*} \| f \|^2_{H^1(\mathbb{S}^3)} &= \| f \|^2_{L^2(\mathbb{S}^3)} + \| \slashgrad f \|^2_{L^2(\mathbb{S}^3)} = \sum_{m \geq 0} (m+1)^3 \, \vertiii{ \hat{f}(\pi_m) }^2.
\end{align*}
It is not difficult to see that the usual higher differentiability Sobolev norms are equivalent to similar expressions,
\[ \| f \|^2_{H^k(\mathbb{S}^3)} \simeq \sum_{m \geq 0} (m+1)^{2k+1} \, \vertiii{ \hat{f}(\pi_m) }^2. \]
We also use the above expression to define the spaces $H^k(\mathbb{S}^3)$ for non-integer $k$. This definition coincides with the usual one via $ H^k(\mathbb{S}^3) = (1-\slashed{\Delta})^{-k/2}(L^2(\mathbb{S}^3))$.

\section{An Auxiliary Lemma}

\begin{lemma} \label{lem:aux_2}
	The series
	\[ S(\beta_0, \alpha, \beta;n) \overset{\text{def}}{=} \sum_{k=1}^{n-1} \frac{1}{(1+|n-2k|)^{1+2\beta_0}} \frac{1}{k^\alpha} \frac{1}{(n-k)^\beta} \]
	converges as $n \to \infty$ if
	\begin{align}
		\label{aux_lemma_condition_1} \alpha + \beta + 2\beta_0 &> 0, \\
		\label{aux_lemma_condition_2} \alpha + \beta &\geq 0, \\
		\label{aux_lemma_condition_3} 1 + 2\beta_0 &\geq 0, \\
		\label{aux_lemma_condition_4} \alpha + 1 + 2\beta_0 &\geq 0, \\
		\label{aux_lemma_condition_5} \beta + 1 + 2\beta_0 &\geq 0.
	\end{align}
	\end{lemma}

\begin{proof}
	By splitting the sum into the two ranges $1 \leq k \leq n /2 $ and $ n /2 < k \leq n - 1$, we obtain
	\begin{align*} S(\beta_0, \alpha, \beta;n) &= \sum_{0 \leq s \leq n-2} (1+s)^{-1-2\beta_0} (n+s)^{-\alpha} (n-s)^{-\beta} \\
	&+ \sum_{0 < s \leq n-2} (1+s)^{-1-2\beta_0} (n-s)^{-\alpha} (n+s)^{-\beta}.	
	\end{align*}
	By symmetry, it is enough to consider the convergence of 
	\[ \tilde{S}(\beta_0, \alpha, \beta;n) \overset{\text{def}}{=} \sum_{s=0}^{n-2} ( 1+ s)^{-1-2\beta_0} (n+s)^{-\alpha} (n-s)^{-\beta}. \]
	
	\subsubsection*{Necessity}
	
	By considering the first and last terms in the above series, we first observe that $\tilde{S}(\beta_0, \alpha, \beta;n) \geq n^{-\alpha-\beta} + 2^{-\alpha - \beta} (n-1)^{-1-\alpha-2\beta_0}$, so
	\begin{equation}
	\label{aux_lemma_necessary_condition_1}
		 \alpha + \beta \geq 0 \quad \text{and} \quad \alpha + 1 + 2\beta_0 \geq 0 
	\end{equation}
	are necessary conditions. By symmetry, we then also need
	\begin{equation}
	\label{aux_lemma_necessary_condition_2}
		\beta + 1 + 2\beta_0 \geq 0.
	\end{equation}
	Together, these imply
	\[ \beta_0 \geq - \frac{1}{2} - \frac{1}{4}(\alpha + \beta). \]
	
	\subsubsection*{Sufficiency} Assume \eqref{aux_lemma_necessary_condition_1} and \eqref{aux_lemma_necessary_condition_2}, and assume further that 
	\begin{equation}
		\label{beta_0_condition}
		\beta_0 \geq - \frac{1}{2}
	\end{equation}
	and
	\begin{equation}
		\label{alpha_beta_beta_0_sum_condition}
		\alpha + \beta + 2\beta_0 > 0.
	\end{equation}
	Suppose first that two of the three exponents vanish. Then \eqref{alpha_beta_beta_0_sum_condition} implies that the remaining exponent is greater than $1$, and convergence follows.
	
	Next, suppose that exactly one of the three exponents vanishes. If $\alpha = 0$, then \eqref{aux_lemma_necessary_condition_1} implies that $\beta > 0$ and $1 + 2\beta_0 > 0$. By H\"older's inequality, convergence follows if $\beta + 2\beta_0 > 0$, which is guaranteed by \eqref{alpha_beta_beta_0_sum_condition}. The other two cases are similar.
	
	Suppose therefore that none of the three exponents vanish. We first treat the case when both $\alpha$ and $\beta$ are positive. Convergence follows trivially if $\alpha > 1$, $\beta > 1$, or $1+ 2\beta_0 > 1$, so suppose first that 
	\[ 0 < \alpha, \, \beta, \, 1+2\beta_0 \leq 1 .\]
	Note that in this case $ -1/2 < \beta_0 \leq 0$ and that \eqref{aux_lemma_necessary_condition_1} and \eqref{aux_lemma_necessary_condition_2} are satisfied. Then we claim that convergence follows by two applications of H\"older's inequality, first to separate the $(1+s)^{-1-2\beta_0}$ term, and then to separate the two remaining terms. Indeed, applying H\"older's inequality once, we wish to find $p, \, q > 1$ such that 
	\[ p (1+ 2\beta_0 ) > 1, \qquad \frac{1}{p} + \frac{1}{q} = 1. \]
	Choose $\epsilon > 0$ such that $\epsilon < 1 + 2\beta_0$, and put 
	\[ p = \frac{1}{1 + 2\beta_0 - \epsilon}. \]
	Then 
	\[ \frac{1}{q} = 1 - \frac{1}{p} = \epsilon - 2\beta_0 \in (0, 1) \]
	for sufficiently small $\epsilon$, so we may choose such a $q > 1$. Now applying H\"older again, we wish to find exponents $r, \, s > 1$ such that 
	\[ \alpha q r > 1, \qquad \beta q s < 1, \quad \text{and} \quad \frac{1}{r} + \frac{1}{s} = 1. \]
	Choose $r> 1$ such that
	\[ \epsilon - 2\beta_0 < \alpha r. \]
	We then seek an $s > 1$ such that
	\[ \frac{1}{s} = 1 - \frac{1}{r} \in (1 - \alpha (\epsilon - 2\beta_0)^{-1} , \, \beta q),  \]
	which is possible if 
	\[ 1 - \frac{\alpha}{\epsilon - 2\beta_0} < \beta q = \frac{\beta}{\epsilon - 2\beta_0} . \]
	This is satisfied if $\alpha + \beta + 2\beta_0 > \epsilon$, which is possible by \eqref{alpha_beta_beta_0_sum_condition} if $\epsilon$ is sufficiently small. 
	
	Finally, suppose that either $\alpha$ or $\beta$ is allowed to be negative. Since we must have $\alpha + \beta \geq 0$, at least of one them must be non-negative, and because $(n+s) \geq (n-s)$, it is enough to treat the case $\beta \geq 0$ and $\alpha \leq \beta$. Let 
	\[ f(n;s) \overset{\text{def}}{=} (1+s)^{-1-2\beta_0} (n+s)^{-\alpha} (n-s)^{-\beta}. \]
	The derivative of $f(n;s)$ with respect to $s$ is given by 
	\[ \frac{\d f}{\d s} = - \frac{f(n;s)}{(n+s)(n-s)(1+s)} q(n,s), \]
	where $q(n,s)$ is the quadratic in $s$ and $n$
	\begin{align*} q(n,s) &= (n^2-s^2)(1+2\beta_0) + (1+s) (\beta(n+s) + \alpha(n-s)) \\
	& = s^2 ( \beta - \alpha - 1 - 2\beta_0) + s((n-1)\alpha + (n+1)\beta) + (1+2\beta_0)n^2 + (\alpha + \beta)n.
	\end{align*}
	Since $\beta_0 \geq -1/2$ and $0 \leq s \leq n-2$,
	\[ q(n,s) \geq 4 (n-1) (1+2\beta_0) +\underbrace{(1+s)}_{\geq 1} ((\alpha+\beta)n + \underbrace{(\beta -  \alpha)s}_{\geq 0}), \]
	so $q(n,s) \geq 0$ $\forall$ $0\leq s \leq n-2$ by \eqref{aux_lemma_necessary_condition_1}, \eqref{beta_0_condition}, and the assumption $\beta - \alpha \geq 0$. So the function $f(n;s)$ is decreasing in $s$ in this case, and hence 
	\[ \sum_{s=0}^{n-2} f(n;s) \leq n^{-(\alpha+\beta)} + I(\beta_0, \alpha, \beta; n), \]
	where 
		\[ I(\beta_0, \alpha, \beta; n) \overset{\text{def}}{=} \int_0^{n-2} (1+s)^{-1-2\beta_0}(n+s)^{-\alpha} (n-s)^{-\beta} \, \d s .  \]
	Making the change of variables $s = n\sigma$, the integral $I(\beta_0,\alpha,\beta;n)$ becomes
	\begin{align*}
		I(\beta_0, \alpha, \beta;n) &= n^{-\alpha - \beta - 2\beta_0} \int_0^{1-\frac{2}{n}} ( n^{-1} +\sigma )^{-1-2\beta_0} (1+\sigma)^{-\alpha} (1-\sigma)^{-\beta} \, \d \sigma \\
		& \la n^{-\alpha - \beta - 2\beta_0} \int_0^{1-\frac{2}{n}} (n^{-1} + \sigma)^{-1-2\beta_0} (1-\sigma)^{-\beta} \, \d \sigma,
	\end{align*}
	where we noted that $I(\beta_0,\alpha, \beta; n)$ is manifestly positive and $(1+\sigma) \simeq 1$ in the domain of integration (for $n \geq 2$). Splitting the range of integration into $[0, \frac{1}{2}]$ and $[\frac{1}{2}, 1 - \frac{2}{n}]$ and noting that in the former range $(1-\sigma) \simeq 1$ and in the latter range $(n^{-1} + \sigma) \simeq 1$, we deduce that $I(\beta_0, \alpha, \beta;n)$ converges as $n\to \infty$ if
	\[ n^{-\alpha - \beta - 2\beta_0} \left( \int_0^{\frac{1}{2}} (n^{-1} + \sigma)^{-1-2\beta_0} \, \d \sigma + \int_{\frac{1}{2}}^{1-\frac{2}{n}} (1-\sigma)^{-\beta} \, \d \sigma \right) \]
	converges. By direct integration, we conclude the following, distinguishing between the cases when the integrals lead to logarithmic terms. If $\beta_0 \neq 0$ and $\beta \neq 1$, we have convergence if $n^{-\alpha - \beta} + n^{-\alpha - 2\beta_0 -1}$ converges, i.e. if $\alpha + \beta \geq 0$ and $\alpha + 2\beta_0 + 1 \geq 0$. These are exactly \eqref{aux_lemma_condition_2} and \eqref{aux_lemma_condition_4}. If $\beta_0 = 0$ and $\beta \neq 1$, we have convergence of $I(\beta_0,\alpha, \beta;n)$ if $n^{-\alpha - \beta} \log n + n^{-\alpha -1}$ converges, i.e. when $\alpha + \beta > 0$ and $\alpha \geq -1$; these are implied by \eqref{aux_lemma_condition_1} and \eqref{aux_lemma_condition_4}. If, on the other hand, $\beta_0 \neq 0$ and $\beta = 1$, convergence follows if $\alpha + 2\beta_0 + 1 > 0$ and $\alpha \geq - 1$; these follow from \eqref{aux_lemma_condition_1} and \eqref{aux_lemma_condition_2}. Finally, if $\beta_0 = 0$ and $\beta = 1$, we have convergence if $\alpha > -1$, which is implied by \eqref{aux_lemma_condition_1}. 	
	\end{proof}

\printbibliography

\end{document}